\documentclass[11pt]{amsart}
\usepackage[utf8]{inputenc}
\usepackage{amsmath}
\usepackage{amssymb, amsthm, amsfonts, amscd, MnSymbol, url}
\usepackage[top=1in, bottom=1in, outer=1.25in, inner=1.25in, heightrounded, marginparwidth=1.15in, marginparsep=.13cm]{geometry}
\usepackage{amsxtra}     
\usepackage{epsfig}
\usepackage{verbatim}
\usepackage{graphicx}
\usepackage[all]{xypic}
\usepackage{todonotes}
\usepackage{marginnote}

\usepackage{mathrsfs}

\usepackage{hyperref}

\usepackage{enumerate}   

\input xy
\xyoption{all}

\def\R{\mathbb R}
\def\Z{\mathbb Z}

\def\bA{\mathbb{A}}
 
\def\bC{\mathbb{C}}

\def\bF{\mathbb{F}}
\def\bG{\mathbb{G}}

\def\bN{\mathbb{N}}

\def\bQ{\mathbb{Q}}
\def\bR{\mathbb{R}}
\def\bS{\mathbb{S}}

\def\bW{\mathbb{W}}

\def\bZ{\mathbb{Z}}
\def\cA{\mathcal{A}}

\def\cD{\mathcal{D}}

\def\cE{\mathcal{E}}

\def\cH{\mathcal{H}}

\def\cL{\mathcal{L}}
\def\cM{\mathcal{M}}

\def\cO{\mathcal{O}}

\def\cR{\mathcal{R}}
\def\cS{\mathcal{S}}

\def\cZ{\mathcal{Z}}

\def\fb{\mathfrak{b}}

\def\fm{\mathfrak{m}}

\def\fp{\mathfrak{p}}
\def\fq{\mathfrak{q}}

\def\fB{\mathfrak{B}}

\def\fE{\mathfrak{E}}
\def\fF{\mathfrak{F}}

\def\fL{\mathfrak{L}}

\def\fP{\mathfrak{P}}

\def\fT{\mathfrak{T}}

\def\Frac{\operatorname{Frac}}

\def\GL{\operatorname{GL}}

\def\Gal{\operatorname{Gal}}

\def\Hom{\operatorname{Hom}}

\def\Im{\operatorname{Im}}
\def\Ind{\operatorname{Ind}}

\def\Isom{\operatorname{Isom}}

\def\Spec{\operatorname{Spec}}

\def\det{\operatorname{det}}
\def\diag{\operatorname{diag}}

\def\dim{\operatorname{dim}}

\def\id{\operatorname{id}}
\def\im{\operatorname{im}}

\def\mod{\operatorname{mod}}

\def\ord{\mathrm{ord}}

\def\sgn{\operatorname{sgn}}

\def\univ{\operatorname{univ}}

\def\ev{\operatorname{ev}} 

\def\eps{\varepsilon}

\newcommand{\beas}{\begin{eqnarray*}}
\newcommand{\eeas}{\end{eqnarray*}}

\newcommand{\wt}[1]{\widetilde{#1}}

\newcommand{\Eisab}{\mu_{G'}}
\newcommand{\ellcan}{\ell_{\mathrm{can}}}
\newcommand{\uo}{\underline{\omega}}

\newcommand{\et}{\text{\it{\'et}}}
\newcommand{\isomto}{\overset{\sim}{\rightarrow}}
\newcommand{\dual}{^\vee}
\newcommand{\hdr}{H_{dR}}
\newcommand{\Fil}{\mathrm{Fil}}
\newcommand{\Gr}{\mathrm{Gr}}
\newcommand{\KS}{\mathrm{KS}}
\newcommand{\ci}{C^\infty}

\newcommand{\ZZ}{\bZ}
\newcommand{\IQ}{\bQ}
\newcommand{\OK}{\mathcal{O}_\cmfield}
\newcommand{\adeles}{\bA}
\newcommand{\hern}{\mathrm{Herm}_n}
\newcommand{\IC}{\bC}

\newcommand{\cmfield}{K}
\newcommand{\realfield}{\cmfield^+}
\newcommand{\Oreal}{\cO_{\realfield}}

\newcommand{\GMplus}{GM_+}

\newcommand{\res}{\mathrm{res}}

\newcommand{\IR}{\bR}
\newcommand{\similitude}{\nu}
\newcommand{\cpct}{\mathcal{U}}
\newcommand{\xzero}{x_0}
\newcommand{\xzeroreduced}{{\overline{x}_0}}

\newcommand{\an}{\mathrm{an}}

\newcommand{\modulispace}{\mathcal{M}}
\newcommand{\uA}{\underline{A}}
\newcommand{\polarization}{\lambda}
\newcommand{\ENDO}{i}
\newcommand{\level}{\alpha}
\newcommand{\PELtuple}{\left(A, \ENDO, \polarization, \level\right)}

\newcommand{\reflex}{E}

\newcommand{\Auniv}{\cA_{\mathrm{univ}}}

\newcommand{\Ekap}{\cE_{\kappa}}

\newcommand{\lambdaiso}{\ell}

\newcommand{\siga}{{a_+}}
\newcommand{\sigb}{{a_-}}
\newcommand{\sigp}{\siga}
\newcommand{\sigm}{\sigb}
\newcommand{\sigpm}{{a_{\pm}}}

\newcommand{\<}{\left\langle}
\renewcommand{\>}{\right\rangle}

\newcommand{\Sord}{{\cS^{\rm ord}}}
\newcommand{\Sordm}{{\cS^{\rm ord}_m}}
\newcommand{\Igusa}{\mathrm{Ig}}
\newcommand{\igusa}{\Igusa}
\newcommand{\fpb}{\overline{\mathbb F}_p}
\newcommand{\Witt}{\mathbb W}

\newcommand{\Levi}{H}
\def\Levin{H}
\newcommand{\zz}{{\mathbb Z}}
\newcommand{\lcan}{\ell_{\rm can}}
\newcommand{\gm}{{\hat{\mathbb G}_m}}
\newcommand{\Ring}{\cR_{\Sord,\xzero}}
\newcommand{\omicron}{a}

\newcommand{\Sordprime}{{\cS}^{'\rm ord}}

\newcommand{\arch}{\Sigma} 
\newcommand{\archK}{\Sigma_K} 
 
\newcommand{\ov}{\overline}
\newcommand{\ul}{\underline}
\newcommand{\ra}{\rightarrow}

\newcommand{\Symm}{\mathfrak{S}}

\newcommand{\toisom}{\buildrel\sim\over\to}

\newcommand{\Mord}{\cM^{\mathrm{ord}}}

\newcommand{\Ig}{\Igusa}

\newtheorem{thm}{Theorem}
\numberwithin{thm}{subsection}
\newtheorem{cor}[thm]{Corollary}

\newtheorem{lem}[thm]{Lemma}
\newtheorem{prop}[thm]{Proposition}
\newtheorem{mainresult}{Main Result}

\theoremstyle{definition}

\newtheorem{defi}[thm]{Definition}
\newtheorem{defn}[thm]{Definition}
 
\theoremstyle{remark}

\newtheorem{rmk}[thm]{Remark}

\numberwithin{equation}{subsubsection}

\title[Differential operators and families]{Differential operators and families of automorphic forms on unitary groups of arbitrary signature}

\author[E. Eischen]{Ellen Eischen}
\thanks{EE's research was partially supported by NSF Grants DMS-1249384 and DMS-1559609.}
\author[J. Fintzen]{Jessica Fintzen}
\thanks{JF's research was partially supported by the Studienstiftung des deutschen Volkes.}
\author[E. Mantovan]{Elena Mantovan}
\thanks{EM's research was partially supported by NSF Grant DMS-1001077.}
\author[I. Varma]{Ila Varma}
\thanks{IV's research was partially supported by a National Defense Science and Engineering Fellowship and NSF Grant DMS-1502834.}

\address{E. Eischen\\
Department of Mathematics\\
University of Oregon\\
Fenton Hall\\
Eugene, OR 97403-1222\\
USA}
\email{eeischen@uoregon.edu}

\address{J. Fintzen\\
 Department of Mathematics\\
 University of Michigan\\
2074 East Hall\\
530 Church Street\\
Ann Arbor, MI 48109\\
 USA}
 \email{fintzen@umich.edu}

\address{E. Mantovan\\
 Department of Mathematics\\
Caltech\\
Pasadena, CA 91125\\
 USA}
 \email{mantovanelena@gmail.com}

\address{I. Varma\\
 Department of Mathematics\\
Columbia University\\
 New York, NY 10027\\
 USA}
 \email{ila@math.columbia.edu}

\begin{document}
\bibliographystyle{amsalpha}

\newpage
\setcounter{page}{1}
\maketitle
\begin{abstract}
In the 1970's, Serre exploited congruences between $q$-expansion coefficients of Eisenstein series to produce $p$-adic families of Eisenstein series and, in turn, $p$-adic zeta functions.  Partly through integration with more recent machinery, including Katz's approach to $p$-adic differential operators, his strategy has influenced four decades of developments.  Prior papers employing Katz's and Serre's ideas exploiting differential operators and congruences to produce families of automorphic forms rely crucially on $q$-expansions of automorphic forms.   

The overarching goal of the present paper is to adapt the strategy to automorphic forms on unitary groups, which lack $q$-expansions when the signature is of the form $(a, b)$, $a\neq b$.  In particular, this paper completely removes the restrictions on the signature present in prior work.  As intermediate steps, we achieve two key objectives.  First, partly by carefully analyzing the action of the Young symmetrizer on Serre--Tate expansions, we explicitly describe the action of differential operators on the Serre--Tate expansions of automorphic forms on unitary groups of arbitrary signature.  As a direct consequence, for each unitary group, we obtain congruences and families analogous to those studied by Katz and Serre.  Second, via a novel lifting argument,
we construct a $p$-adic measure taking values in the space of $p$-adic automorphic forms on unitary groups of any prescribed signature.  We relate the values of this measure to an explicit $p$-adic family of Eisenstein series.  One application of our results is to the recently completed construction of $p$-adic $L$-functions for unitary groups by the first-named author, Harris, Li, and Skinner.
\end{abstract}

\tableofcontents
\section{Introduction}

\subsection{Motivation and context}\label{strategy-innovations}

\subsubsection{Influence of a key idea of Serre about congruences}\label{history-section}
J.-P. Serre's idea to exploit congruences between Fourier coefficients of Eisenstein series
 to construct certain $p$-adic zeta-functions continues to have a far-reaching impact.  His strategy has led to numerous developments, partly through integration with more recent machinery.  For example, his approach is seen in work on the Iwasawa Main Conjecture (e.g. in \cite{SkUr}).  Emblematic of the reach of Serre's idea to interpolate Fourier coefficients of Eisenstein series, his $p$-adic families of Eisenstein series also occur even in homotopy theory, as the {\it Witten genus}, an invariant of certain manifolds \cite{hopkinsICM, AHR}.

Serre's idea in \cite{serre} has been employed in increasingly sophisticated settings.  J. Coates and W. Sinnott extended it to construct $p$-adic $L$-functions over real quadratic fields \cite{coates-sinnott}, followed by P. Deligne and K. Ribet over totally real fields \cite{DR}.  Developing it further, N. Katz handled CM fields $K$ (when $p$ splits in $K$), using congruences between Fourier coefficients of Eisenstein series in the space of Hilbert modular forms \cite{kaCM}.  Using Katz's Eisenstein series, H. Hida produced $p$-adic $L$-functions of families of ordinary cusp forms \cite{hidaLfcn}, leading to A. Panchishkin's $p$-adic $L$-functions of non-ordinary families \cite{pa}.

The $p$-adic families of Eisenstein series on unitary groups of signature $(n,n)$ in \cite{apptoSHL, apptoSHLvv, emeasurenondefinite} and related families of automorphic forms for arbitrary signature in Theorem \ref{measurethm} of this paper play a key role in the recent construction of $p$-adic $L$-functions for unitary groups \cite{EHLS}.  (The related approach proposed in \cite{HLS}, on which \cite{EHLS} elaborates, also inspired work in \cite{bou, EW, FUH, ZL, XW}.)
These $p$-adic families also conjecturally give an analogue of the Witten genus, at least for signature $(1, n)$ \cite{beh}.  
\subsubsection{An inspiration for four decades of innovations and the necessity of more}

Most of the four decades of developments in Section \ref{history-section} require increasingly sophisticated methods, even though the overarching strategy (``find a family of Eisenstein series, observe congruences, relate to an $L$-function'') is well-established.  The devil is in the details.  We now highlight three ingredients from the above constructions most relevant to the details of the present work: (1) $q$-expansions; (2) differential operators; (3) Eisenstein series.

\medskip

\noindent {\bf (1) $q$-expansions.}
All prior papers employing Serre's idea to exploit congruences between Eisenstein series rely crucially on the $q$-expansions of automorphic forms.   The key goal of the present paper is to extend the aforementioned strategies to automorphic forms on unitary groups, which lack $q$-expansions when the signature is not $(n,n)$.  In their place, we use {\it Serre--Tate expansions} (or {\it $t$-expansions}), expansions at ordinary CM points (whose structure leads to a natural choice of coordinates, {\it Serre--Tate coordinates}) and the Serre--Tate Expansion Principle \cite[Theorem 5.14, Proposition 5.5, Corollary 5.16]{CEFMV}.

\medskip

\noindent {\bf (2) Differential operators.} A key innovation of Katz in \cite{kaCM} is the construction of $p$-adic differential operators (generalizing the Maass--Shimura operators studied extensively by Harris and Shimura \cite{hasv, sh, Shimura, shclassical}) and a description of their action on $q$-expansions, when $p$ splits in the CM field.  Lacking $q$-expansions, we compute the action of differential operators on Serre--Tate expansions and, as a consequence, produce congruences and families of $p$-adic automorphic forms.  Our work builds on \cite{E09, EDiffOps, CEFMV, KaST, brooks} and requires careful analysis of the action of Schur functors (in particular the Young symmetrizer) on Serre--Tate expansions.  (In a different direction, E. Goren and E. de Shalit recently constructed $p$-adic differential operators for signature $(2,1)$ with $p$ inert \cite{DSG}.)  

\medskip

\noindent{\bf (3) Eisenstein series.} The constructions in Section \ref{history-section} rely on congruences between $q$-expansion coefficients of Eisenstein series.  For unitary groups of arbitrary signature, we compensate with explicit computation of the action of the Young symmetrizer on Serre--Tate coordinates.  Also applying a novel lifting argument to the Eisenstein series on unitary groups of signature $(n,n)$ constructed in \cite{apptoSHL, apptoSHLvv} independently of Serre--Tate coordinates, we also construct explicit $p$-adic families for arbitrary signature.

\subsection{This paper's main results, innovations, and connections with prior work}\label{innovation-section}

As noted above, a key accomplishment of this work is that it produces families without needing $q$-expansions and thus is applicable to unitary groups of all signatures.  The results and techniques in this paper carry over to the automorphic forms in other papers extending Serre's strategy (i.e. Siegel modular forms, Hilbert modular forms, and modular forms) but are unnecessary in those settings (since they have $q$-expansions).

As a consequence of the work in the first sections of this paper, we finish the problem of constructing $p$-adic families sufficient for the $p$-adic $L$-functions in \cite{EHLS}, completely eliminating conditions on signatures.  We also expect our results on Serre--Tate expansions to have applications to the extension to the setting of unitary groups of the results of A. Burungale and Hida on $\mu$-invariants \cite{BH}.

\subsubsection{Three main results}
Main Results 1, 2, and 3 below rely on a careful application of a combination of arithmetic geometric, representation theoretic, and number theoretic tools.  We denote by $V$ the space of $p$-adic automorphic forms (global sections of a line bundle over the Igusa tower, as defined by Hida in \cite{hida}) on a unitary group $G$.

\begin{mainresult}[summary of Theorem \ref{congruence} and Corollary \ref{cong-coro}]\label{MR-1}
For each classical or $p$-adic weight $\kappa$ (viewed as a character) meeting mild conditions, there is a $p$-adic differential operator $\Theta^\kappa$ acting on $V$, with the property that if the weight of $f\in V$ is $w$, then the weight of $\Theta^\kappa f$ is $w\cdot \kappa$, and if $\kappa\equiv\kappa'\mod p^e$ for some $e$, then $\Theta^\kappa f\equiv\Theta^{\kappa'}f\mod p^e$.  As a consequence, one can use the operators to obtain $p$-adic families of forms (which are closely related to certain $\ci$-automorphic forms, e.g.\ those appearing in Main Result \ref{MR-3}).
\end{mainresult}

\begin{mainresult}[summary of Theorem \ref{Theta} and Corollary \ref{cordef}]\label{MR-2}
While constructing and explicitly describing the action of $\Theta^\kappa$ on Serre--Tate expansions (also called $t$-expansions), we compute the precise polynomials (in the proof of Proposition \ref{cong1}) by which the coefficients in the expansion are multiplied upon applying the differential operators.
\end{mainresult}

\begin{mainresult}[summary of Theorem \ref{measurethm}]\label{MR-3}
There is a $p$-adic measure taking values in $V$ and providing an explicit family of $p$-adic automorphic forms closely related to the $\ci$ Eisenstein series studied by Shimura in \cite{sh}.
\end{mainresult}

\subsubsection{Methods}
The construction of the differential operators builds on earlier results on $p$-adic differential operators in \cite{kaCM} (for Hilbert modular forms), \cite{EDiffOps, emeasurenondefinite} (for unitary groups of signature $(n,n)$ and pullbacks to products of definite unitary groups), and \cite{pa-maass} (for Siegel modular forms).  Unlike in those earlier cases, though, the lack of $q$-expansions in the case of unitary groups not of signature $(n,n)$ necessitates modifying the approach of those papers.  Instead, we take expansions at ordinary CM points ({\it Serre--Tate expansions}) and apply the Serre--Tate Expansion Principle \cite[Theorem 5.14, Proposition 5.5, Corollary 5.16]{CEFMV}.  Also, unlike earlier constructions, by employing Hida's density theorem, we extend the action of the operators to $p$-adic (not necessarily classical) weights.

Our ability to establish congruences among differential operators depends on appropriately choosing the $p$-adic integral models for the algebraic representations associated with dominant weights.  In particular, our models are slightly different from those considered in the work of Hida \cite{hida}, and our construction relies on the theory of Schur functors and projectors.  The congruences follow from a careful analysis of the action of Schur functions (and especially the generalized Young symmetrizer) and rely on the description in Main Result \ref{MR-2}.  

In Section \ref{pullbacks-section}, we also extend Main Results \ref{MR-1} and \ref{MR-2} to the case of pullbacks from a Shimura variety to a subvariety.
While Main Results \ref{MR-1} and \ref{MR-2} focus on the description of the operators on Serre--Tate expansions, the precision with which we work out details for Serre--Tate expansions allows us also to transfer some of our results to devise a novel lifting argument (in the proof of Theorem \ref{measurethm}) concerning only $q$-expansions that produces explicit families of automorphic forms, summarized in Main Result \ref{MR-3}.  
(The key idea is to apply a lifting argument, together with the description of the action of the differential operators and pullbacks developed in Section \ref{pullbacks-section} and in the proof of Proposition \ref{cong1} to the Eisenstein series constructed in \cite{apptoSHL, apptoSHLvv}.)  These families feed into the machinery of $p$-adic $L$-functions in \cite{EHLS}.

\begin{rmk}
Although there are no $q$-expansions in the setting of unitary groups of arbitrary signature, these operators can naturally be viewed as the incarnation of Ramanujan's operator $q\frac{d}{dq}$ in this setting.  The families that can be obtained by applying such operators are broader than what can be obtained by tensoring with powers of a lift of the Hasse invariant, since our construction allows, for example, non-parallel weights.  
\end{rmk}

\begin{rmk}
It would also be beneficial to have a $p$-adic Fourier--Jacobi expansion principle for unitary groups of arbitrary signature (which appears to be possible to state and prove - via a lengthy, technical argument -  building on recent arithmetic geometric developments, e.g. \cite{lan}).  This would provide an alternate but ultimately more direct route (modulo the necessity of first proving such an expansion principle) to the construction of families.  On the other hand, we also expect our work with the Serre--Tate expansions themselves to be useful in other applications, e.g. an extension of Burungale and Hida's work in \cite{BH}.
\end{rmk}

\subsection{Structure of the paper}\label{structure-section}

Section \ref{background-section} introduces our setup and recalls key facts about unitary Shimura varieties, the Igusa tower, and $p$-adic and classical automorphic forms (following \cite[Sections 2 and 3]{CEFMV}).  It also provides necessary results on $t$-expansions from \cite{CEFMV} and an overview of Schur functors, on which our computations rely crucially.

Sections \ref{diffop-sec} and \ref{local-section} give global and local descriptions, respectively, of the differential operators.  In particular, Section \ref{diffop-sec} discusses differential operators of integral (classical) weights that act on the automorphic forms introduced in Section \ref{background-section}.   Using Schur functors, we build these operators from the Gauss--Manin connection and the Kodaira--Spencer morphism.  In Section \ref{local-section}, via a careful computation of the action of the Young symmetrizer, we describe the action of the differential operators on Serre--Tate coordinates. 

In Section \ref{mainresultsops-section}, we use the description of the action of the differential operators on $t$-expansions, and Hida's density theorem, to prove the operators extend to the whole space of $p$-adic automorphic forms.   We then establish congruences among operators of congruent weights, which we interpolate to differential operators of $p$-adic weights on the space of $p$-adic automorphic forms, leading to Theorem \ref{congruence} and Corollary \ref{cong-coro} (summarized in Main Result \ref{MR-1} above) and Theorem \ref{Theta} and Corollary \ref{cordef} (summarized in Main Result \ref{MR-2} above).

Section \ref{pullbacks-section} describes the behavior of the differential operators with respect to restriction from one unitary group to a product of two smaller unitary groups.  Restrictions of $p$-adic automorphic forms play a crucial role in the construction of $p$-adic $L$-functions in \cite{EHLS}.

Section \ref{families-section} constructs $p$-adic families of automorphic forms on unitary groups of arbitrary signature, by applying a novel lifting strategy and our $p$-adic differential operators to restrictions of $p$-adic families of Eisenstein series from \cite{apptoSHL, apptoSHLvv}.  Theorem \ref{measurethm} (summarized in Main Result \ref{MR-3} above) produces a $p$-adic measure taking values in the space of $p$-adic automorphic forms of arbitrary signature related to the given family of Eisenstein series.  This result is in turn used in the construction of $p$-adic $L$-functions in \cite{EHLS}.

\subsection{Notation and conventions}\label{notation-section}

Fix a totally real number field $\cmfield^+$ of degree $r$ and an imaginary quadratic extension $K_0$ of $\IQ$.  Define $\cmfield$ to be the compositum of $\cmfield^+$ and $K_0$. Additionally, we will fix a positive integer $n$, and a rational prime $p > n$ that splits completely in $\cmfield/\bQ$. If $K$ is an imaginary quadratic field, we put further restrictions when $n = 2$ (see Remark 2.3.2).

The above assumptions ensure the following:
\begin{itemize}
\item Our unitary group at $p$ is a product of (restrictions of scalars of) general linear groups 
\item The Shimura varieties of prime-to-$p$ level we consider have smooth integral models (with moduli interpretations) after localizing at $p$
\item Sections of automorphic bundles on open Shimura varieties coincide with those on their compactifications, by Koecher's principle.
\item The ordinary locus of the reduction modulo $p$ is not empty
\item We can $p$-adically interpolate our differential operators (see Proposition \ref{cong1} and Theorem \ref{congruence}).
\end{itemize}

We now discuss some notation used throughout the paper. For any field $L$, we denote the ring of integers in $L$ by $\cO_L$. We use $\bA$ to denote the adeles over $\bQ$, and we write $\bA^\infty$ (resp. $\bA^{\infty,p}$) to denote the adeles away from the archimedean places (resp. the archimedean places and $p$). For our CM field $\cmfield$, let $c$ denote complex conjugation, i.e. the generator of $\Gal(\cmfield/\cmfield^+)$.  We denote by $\arch$ the set of embeddings of $\cmfield^+$ into $\overline{\bQ}_p$, and we denote by $\archK$ the set of $\overline{\bQ}_p$-embeddings of $\cmfield$.  Additionally, fix a {\em CM type} of $K$, i.e. for each $\tau \in \Sigma$ choose exactly one $K$-embedding $\tilde{\tau}$ extending $\tau$, and abusing notation, identify the set of $\tilde{\tau}$ with $\Sigma$. Under this identification, note that $\Sigma \cup \Sigma^c = \Sigma_K$. 
Additionally, fix an isomorphism $\imath_p:\bC \stackrel{\sim}{\rightarrow} \overline{\bQ}_p$, and let $\Sigma_{\infty} = \imath_p^{-1} \Sigma$ and $\Sigma_{K,\infty} = \imath_p^{-1} \Sigma_K$. We will often identify $\Sigma_{\infty}$ with $\Sigma$ and  $\Sigma_{K,\infty}$ with $\Sigma_K$ via the above isomorphism without further mentioning.

The reflex field associated to our Shimura varieties will typically be denoted by $\reflex$ (with subscripts to denote different reflex fields in Section 6). Additionally, define the primes above $p$ using the decomposition of $p \cO_K = \prod_{i=1}^r \fP_i\fP_i^c$ where $\fp_i = \fP_i\fP_i^c$ are the primes above $p$ in $\cO_{K^+}$.

We denote the dual of an abelian scheme $A$ by $A\dual$.  We also denote the dual of a module $M$ by $M\dual$. Given schemes $S$ and $T$ over a scheme $U$, we denote the scheme $S\times_U T$ by $S_T$.  {When no confusion is likely to arise, we sometimes use the same notation for a sheaf of modules and a corresponding module (e.g. obtained by localizing at a point).}

For any ring $R$, we denote by $M_{n\times n}(R)$ the space of $n\times n$ matrices with entries in $R$, and we denote by $\hern(\cmfield)$ the space of Hermitian matrices inside $M_{n\times n}(\cmfield)$.

\section{Background and setup}\label{background-section}
In this section, we recall facts about Shimura varieties, automorphic forms, and $p$-adic automorphic forms that will play a key role in the rest of the paper.  Most of this material is covered in detail in \cite[Sections 2.1 and 2.2]{CEFMV}.  Like in \cite{CEFMV}, the definitions of the PEL data and moduli problems follow \cite[Sections 4 and 5]{kottwitz} and  \cite[Sections 1.2 and 1.4]{lan}.

\subsection{Unitary groups and PEL data}\label{PEL-data}

 By a {\it PEL datum}, we mean a tuple $\left(\cmfield, c, L, \langle, \rangle, h\right)$ consisting of
\begin{itemize}
\item the CM field $\cmfield$ equipped with the involution $c$ introduced in Section \ref{notation-section},
\item an $\OK$-lattice $L$, i.e. a finitely generated free $\ZZ$-module with an action of $\OK$,
\item a non-degenerate Hermitian pairing $\langle\cdot,\cdot\rangle:L\times L\to \mathbb{Z}$ satisfying $\langle k\cdot v_1,v_2 \rangle=\langle v_1,k^c\cdot v_2 \rangle$ for all $v_1, v_2\in L$ and $k\in \OK$, 
\item an $\mathbb{R}$-algebra endomorphism \[h:\mathbb{C}\to \mathrm{End}_{\OK\otimes_\mathbb{Z}\mathbb{R}}(L\otimes_{\mathbb{Z}}\mathbb{R})\] such that $(v_1,v_2)\mapsto \langle v_1,h(i)\cdot v_2\rangle$ is symmetric and positive definite and such that $\langle h(z) v_1, v_2\rangle = \langle v_1, h(\overline{z})v_2 \rangle$.
\end{itemize}
Furthermore, we require:
\begin{itemize}
\item{$L_p:=L\otimes_{\mathbb{Z}}\mathbb{Z}_p$ is self-dual under the alternating Hermitian pairing $\langle \cdot,\cdot\rangle_p$ on $L\otimes_{\mathbb{Z}}\mathbb{Q}_p$.} 
\end{itemize}
Given a PEL datum $\left(\cmfield, c, L, \langle, \rangle, h\right)$, we associate algebraic groups $GU= GU(L, \langle, \rangle)$, defined over $\ZZ$, whose $R$-points (for any $\ZZ$-algebra $R$) are given by
\begin{align*}
GU(R)&:=\left\{(g,\similitude)\in \mathrm{End}_{\OK\otimes_\mathbb{Z}R}(L\otimes_{\mathbb{Z}}R) \times R^\times\mid\langle g\cdot v_1,g\cdot v_2\rangle =\similitude\langle v_1,v_2\rangle\right\}\\
U(R)&:=\left\{g\in \mathrm{End}_{\OK\otimes_\mathbb{Z}R}(L\otimes_{\mathbb{Z}}R) \mid\langle g\cdot v_1,g\cdot v_2\rangle =\langle v_1,v_2\rangle\right\}
\end{align*} 
Note that $\similitude$ is called the {\it similitude factor}.
Additionally, we can define the $\IR$-vector space equipped with an action of $\cmfield:$ 
\begin{align*}
V:=L\otimes_{\ZZ}\IR.
\end{align*}
The endomorphism $h_{\bC} = h\times_{\mathbb{R}}\mathbb{C}$ gives rise to a decomposition $V_{\bC} := V\otimes_{\IR}\mathbb{C}=V_1\oplus V_2$ (where $h(z) \times 1$ acts by $z$ on $V_1$ and by $\bar{z}$ on $V_2$).  The {\it reflex field} $E$ of $(V, \langle, \rangle, h)$ is the field of definition of the $GU(\bC)$-conjugacy class of $V_1$.

We have further decompositions $V_1 = \oplus_{\tau\in\arch_K} V_{1, \tau}$ and $V_2 = \oplus_{\tau\in\arch_K}V_{2, \tau}$ 
induced from the decomposition of $K \otimes_{\bQ} \bC = \oplus_{\tau \in \Sigma_{K}} \bC$, where only the $\tau$-th $\bC$ acts nontrivially on $V_{\tau} = V_{1,\tau} \oplus V_{2,\tau}$, and it acts via the standard action on $V_{1,\tau}$ and via conjugation on $V_{2,\tau}$.  The {\it signature} of $\left(V, \langle, \rangle, h\right)$ is the tuple of pairs $\left(\siga_\tau, \sigb_\tau\right)_{\tau\in\arch_K}$ where $\siga_\tau = \dim_\IC V_{1, \tau}$ and $\sigb_\tau = \dim_\IC V_{2, \tau}$ for all $\tau\in\archK$. 
The sum $\siga_\tau+\sigb_\tau$ is independent of $\tau\in \arch_K$, and so we define
$$n:=\siga_\tau+\sigb_\tau.$$
(Note that $a_{+\tau^c} = a_{-\tau}$.)  Finally, we define an algebraic group 
$$\Levin :=\prod\limits_{\tau \in \arch_{}} \GL_{\siga_{\tau}} \times \GL_{\sigb_{\tau}}$$ over $\bZ$.  Note that $\Levin(\bC)$ can be identified with the Levi subgroup of $U({\bC})$ that preserves the decomposition $V_\bC=V_1\oplus V_2$. {Additionally, we will denote the {diagonal} maximal torus of $H$ by $T$, and the unipotent radical {of the Borel subgroup of upper triangular matrixes} by $N$.}

\subsection{PEL moduli problem and Shimura varieties}\label{PELmoduli-section}
We now introduce the Shimura varieties associated to a given PEL datum $(K,c,L,\langle,\rangle, h)$.  We will restrict our attention to the integral models (defined over $\cO_E \otimes \bZ_{(p)}$) of such PEL-type unitary Shimura varieties that have prime-to-$p$ level structure and good reduction at $p$. 

Let $\cpct\subset GU(\mathbb{A}^\infty)$ be an open compact subgroup. We assume $\cpct=\cpct^p\cpct_p$ is neat (as defined in \cite[Definition 1.4.1.8]{lan}) and that $\cpct_p\subset GU(\mathbb{Q}_p)$ is hyperspecial. 
Consider the moduli problem $(S,s)\mapsto \left\{\left(A,\ENDO,\polarization, \level\right)/\sim\right\}$
which assigns to every connected, locally noetherian scheme $S$ over $\cO_E\otimes \mathbb{Z}_{(p)}$ together with a geometric point $s$ of $S$, the set of equivalence classes of tuples $\left(A,\ENDO,\polarization, \level\right)$, where:
\begin{itemize}
\item $A$ is an abelian variety over $S$ of dimension $g := nr = n[K^+:\bQ]$, 
\item $i:\cO_{K,(p)}\hookrightarrow (\mathrm{End}(A))\otimes_{\mathbb{Z}}\mathbb{Z}_{(p)}$ is an embedding of $\mathbb{Z}_{(p)}$-algebras,
\item $\lambda: A\to A^\vee$ is a prime-to-$p$ polarization satisfying $\lambda \circ i({k^c})=i(k)^\vee \circ \lambda$  for all $k\in \cO_K$,
\item $\alpha$ is a $\pi_1(S,s)$-invariant $\cpct^p$-orbit of $K\otimes_{\mathbb{Q}}\mathbb{A}^{p,\infty}$-equivariant isomorphisms \[L\otimes_{\mathbb{Z}}\mathbb{A}^{p,\infty} \toisom V^{p}A,\] which takes the Hermitian pairing $\langle \cdot,\cdot \rangle$ on $L$ to an $(\mathbb{A}^{p,\infty})^\times$-multiple of the $\lambda$-Weil pairing $(\cdot,\cdot)_{\lambda}$ on $V^pA$ (the Tate module away from $p$). 
\end{itemize}
In addition, the tuple $\PELtuple$ 
must satisfy Kottwitz's \emph{determinant condition}: \[\mathrm{det}_\IC(\OK|V_1)=\mathrm{det}_{\cO_S}(\OK|\mathrm{Lie}A).\]

Two tuples $\PELtuple \sim \left(A',\ENDO',\polarization',\level'\right)$ are equivalent if there exists a prime-to-$p$ isogeny $A\to A'$ taking $\ENDO$ to $\ENDO'$, $\polarization$ to a prime-to-$p$ rational multiple of $\polarization'$ and $\level$ to $\level'$.

This moduli problem is representable by a smooth, quasi-projective scheme $\cM_\cpct$ over $\cO_E\otimes \mathbb{Z}_{(p)}$.  (See \cite[Corollary 7.2.3.10]{lan}.)  If we allow $\cpct^p$ to vary, the inverse system consisting of $\cM_\cpct$ has a natural action of $GU(\mathbb{A}^{\infty,p})$ (i.e., $g\in GU(\mathbb{A}^{\infty, p})$ acts by precomposing the level structure $\alpha$ with it). 
For any scheme $S$ over $\Spec(\cO_E\otimes \mathbb{Z}_{(p)})$, we put
\begin{align*}
\cM_{\cpct,S}:=\cM_{\cpct}\times_{\cO_{E,(p)}}S.
\end{align*} 
When $S = \Spec(R)$ for a ring $R$, we will often write $\cM_{\cpct,R}$ instead of $\cM_{\cpct,\Spec(R)}$.  

We denote by $\cM_\cpct^{\an}$ the complex manifold of $\IC$-valued points of $\cM_\cpct$ and by $\cM_\cpct\left(\ci\right)$ the underlying $\ci$-manifold.  Given a sheaf $\mathcal{F}$ on $\cM_\cpct$, we denote by $\mathcal{F}\left(\ci\right)$ the sheaf on $\cM_\cpct\left(\ci\right)$ obtained by tensoring $\mathcal{F}$ with the $\ci$-structural sheaf on $\cM_\cpct\left(\ci\right)$. {In the sequel, we fix the level $\cpct$  and suppress it from the notation.}

\subsection{Automorphic forms}\label{classical-aut-forms}
Let $\bW := W(\overline{\bF}_p)$ denote the ring of Witt vectors; note that $\Frac(\bW)$ contains all embeddings of $K \hookrightarrow \overline{\bQ}_p$, due to our assumptions on $p$.  Let $\pi: \Auniv =  \left(A,\ENDO,\polarization, \level\right)^{\univ}\rightarrow\cM_{\bW}$ 
 denote the universal abelian scheme. 
Define $\uo_{\Auniv/\cM} = \pi_{\ast}\underline{\Omega}_{\Auniv/\cM}$ as the pushforward along the structure map of the sheaf of relative differentials.  It is a locally free sheaf of rank $nr$,  equipped with the structure of an $\cO_K \otimes \bW$-module induced by the action of $\cO_K$ on $\Auniv$.  Hence, we obtain the decomposition:
\begin{align}\label{decomptau1}
 	\uo_{\Auniv/\cM} = \bigoplus_{\tau \in \arch} (\uo_{{\Auniv/\cM},\tau}^+ \oplus \uo_{{\Auniv/\cM},\tau}^-)
\end{align}
where $\uo_{{\Auniv/\cM},\tau}^{\pm}$ has rank $a_{\pm \tau}$ and an element $x \in \cO_K$ acts on $\uo_{{\Auniv/\cM},\tau}^+$ (resp. $\uo_{{\Auniv/\cM},\tau}^{-}$) via $\tau(x)$ (resp. $\tau^c(x)$). We can then define $\cE_\cpct = \cE$ as the sheaf:
 	 $$\cE = \bigoplus_{\tau\in \arch}\underline{\Isom}_{\mathcal{O}_{\cM}}\left((\mathcal{O}_{\cM})^{\sigp_\tau}, \uo_{\Auniv/{\cM}, \tau}^+\right) 
  \oplus \bigoplus_{\tau\in \arch}\underline{\Isom}_{\mathcal{O}_{\cM}}\left((\mathcal{O}_{\cM})^{\sigm_\tau}, \uo_{\Auniv/{\cM}, \tau}^-\right).$$
Note there is a (left) action of $H$ on $\cE$ arising from the action of $\GL_{a_{\pm \tau}}$ on $\underline{\Isom}\left(\left(\mathcal{O}_{\cM}\right)^{a_{\pm\tau}}, \uo_{\Auniv/{\cM}, \tau}^{\pm}\right)$ for all $K$-embeddings $\tau \in \Sigma$.

Consider an algebraic representation $\rho$ of $\Levin$ (over $\bW$) into a finite free $\bW$-module $M_{\rho}$.
 For any such $\rho$, we define the sheaf 
$\cE_{\rho}=\cE_{\cpct,\rho} := \cE \times^{\rho} M_{\rho}$, i.e. for each open immersion $\Spec {R} \hookrightarrow \cM$, set $\cE_{\rho}(R):=\left(\cE(R)\times M_{\rho}\otimes_\bW R\right)/\left(\lambdaiso, m\right)\sim\left(g\lambdaiso, \rho({ }^tg^{-1})m\right)$.

An {\em automorphic form of weight $\rho$} defined over a {$\Witt$-algebra} $R$ 
		 is a global section of the sheaf $\cE_{\rho}$ on $\cM_{R}$.

\begin{rmk} Usually, automorphic forms are defined as sections on a compactification of $\cM_{R'}$.  By Koecher's principle the two definitions are equivalent, except when $\Sigma$ consists only of one place $\tau$ and $(a_{+\tau},a_{-\tau})=(1,1)$.  For the remainder of the paper, we exclude this case.  \end{rmk}

\subsection{Standard representation, highest weights, and Schur functors}
We briefly recall some useful facts about $H$ from the theory of algebraic {(rational)}  representations of linear algebraic groups. 

\subsubsection{Highest weights of an algebraic representation of $H$}
The irreducible algebraic representations of $\Levin =\prod_{\tau \in \arch} \GL_{\siga_{\tau}} \times \GL_{\sigb_{\tau}}$ over {any algebraically closed field of characteristic 0} {(up to isomorphism)}  are in one-to-one correspondence with the dominant weights of its diagonal torus $T = \prod_{\tau \in \arch} T_{\siga_\tau} \times T_{\sigb_{\tau}}$. 

For $1 \leq i \leq n$, let $\eps^{\tau}_i$ in $X(T):=\Hom_{{\overline{\bQ}_p}}(T,\bG_m){=\Hom_{\bZ_p}(T,\bG_m)}$ be the character defined by
	\begin{eqnarray*}
	\eps^{\tau}_i: T({\overline{\bQ}_p}) = \prod_{\sigma \in \arch} T_{\siga_\sigma}(\overline{\bQ}_p) \times T_{\sigb_{\sigma}}(\overline{\bQ}_p) &\rightarrow &\bG_m(\overline{\bQ}_p) \vspace{2pt}\\ \eps^{\tau}_i(\diag(\gamma^{\sigma}_{1,1}, \hdots, \gamma^{\sigma}_{n,n})_{\sigma \in \arch}) &=& \gamma^\tau_{{i,i}}.
	\end{eqnarray*}
These characters form a basis of the free $\bZ$-module $X(T)$. We choose $\Delta=\{ \alpha^{\tau}_i := \eps^{\tau}_i - \eps^{\tau}_{i+1} \}_{\tau \in \arch, 1 \leq i < n , i \neq \siga_\tau}$ as a basis for the root system of $\Levin$. 
The set of all {\em dominant weights} of $T$ with respect to $\Delta$ is $X(T)_+=\{ \kappa \in X(T) \, | \, \< \kappa, \check{\alpha} \> \geq 0 \, \forall \alpha \in \Delta  \}$. Using the basis $\{\eps_{i}^{\tau} : \tau \in \Sigma, 1 \leq i \leq n \}$ of $X(T)$, we identify:
	$$X(T)_+ \cong \{(\kappa^{\tau}_1,\hdots, \kappa^{\tau}_{n})_{\tau \in \arch} \in  \prod\limits_{\tau \in \arch} \bZ^{n} \ : \kappa^{\tau}_i \geq \kappa^{\tau}_{i+1}\ \forall i \neq\siga_\tau\},$$
where $\kappa = \left(\kappa^\tau\right)_{\tau\in\Sigma}$ and $\kappa^\tau = \prod_{i}(\varepsilon^\tau_i)^{\kappa^\tau_i}$.  For each  dominant weight $\kappa$, $\rho_{\kappa}: \Levin{_{\overline{\bQ}_p}} \rightarrow M_{\rho}$ denotes {an} irreducible algebraic representation of highest weight $\kappa$.  (See, for example, \cite[Part~II.  Chapter~2]{Jantzen}.)

\subsubsection{Schur functors} \label{section-Schur} {We briefly recall the construction of Schur functors, adapted to our setting. (We refer to \cite[Sections 4.1 and 15.3]{FH} for the usual definitions).}

For $\kappa$ a positive dominant weight, i.e. $\kappa=(\kappa^{\tau}_1,\dots,\kappa^{\tau}_n)_{\tau \in \Sigma} \in X(T)_+$ satisfying $\kappa^{\tau}_i \geq \kappa_{i+1}^{\tau} \geq 0$ for all $\tau \in \Sigma$ and $i \neq \siga_\tau$, we write $d^{\tau+}_\kappa=|\kappa^{\tau+}| :=\sum_{i=1}^{\siga_\tau}\kappa^{\tau}_i$, $d^{\tau-}_\kappa=|\kappa^{\tau-}| :=\sum_{i=\siga_\tau+1}^n\kappa^{\tau}_i$, 
 and  regard {$\kappa^{\tau+} =(\kappa^{\tau}_1,\dots,\kappa^{\tau}_{\siga_\tau}) $ and $\kappa^{\tau-} =(\kappa^{\tau}_{\siga_\tau},\dots,\kappa^{\tau}_n) $ as a partition of $d^{\tau\pm}_\kappa$}.  When there is an integer $k$ such that $\kappa_i^\tau=k$ for all $i$ and $\tau$, we denote $\underline{k}:=\kappa$.  

To each $\kappa^{\tau{\pm}}$, there is an associated {\em Young symmetrizer} $c^{\tau{\pm}}_\kappa\in\bZ[\Symm_{d^{{\tau\pm}}_\kappa}]$ in the group algebra of the symmetric group $\Symm_{d^{{\tau\pm}}_\kappa}$ on $d^{{\tau\pm}}_{\kappa}$ symbols.  If $V$ is any module over a ring $R$, we let $\Symm_{d^{{\tau\pm}}_\kappa}$ act on the ${d^{{\tau\pm}}_\kappa}$-th tensor power $V^{\otimes {d^{{\tau\pm}}_\kappa}}$ {on the right} by permuting factors. This action extends to give a right-$\bZ[\Symm_{{d^{{\tau\pm}}_\kappa}}]$-module structure on $V^{\otimes {d^{{\tau\pm}}_\kappa}}$.

We define the \textit{$\kappa^{{\tau\pm}}$-Schur functor} on the category of $R$-modules 
\[\bS_{\kappa^{{\tau\pm}}}(V):=V^{\otimes d_{\kappa}^{{\tau\pm}}}\cdot c^{{\tau\pm}}_\kappa\subset V^{\otimes d_{\kappa}^{{\tau\pm}}}.\]

{We now assume $R$ {is a {$\bZ_p$}-algebra or an algebraically closed field of characteristic 0.} 	Then  for each $\kappa^{{\tau\pm}}$,  $\bS_{\kappa^{{{\tau\pm}}}}(R^{\sigpm_\tau})$ is {an} irreducible representation of $\GL_{\sigpm_\tau}$ with highest weight $\kappa^{{\tau\pm}}$ (\cite[Proposition 15.15 {and Proposition~15.47}]{FH}).

 Let $V=\bigoplus_{\tau \in \Sigma}\left(V^{+,\tau} \oplus V^{-,\tau}\right)$ be an $R$-module, together with such a decomposition.  We define the \textit{$\kappa$-Schur functor}, for $\kappa$ a positive dominant weight,
by 
\begin{eqnarray*}
\bS_\kappa(V) &:=& \boxtimes _{\tau \in \arch} \left(\bS_{\kappa^{\tau+}}(V^{+,\tau}) \boxtimes  \bS_{\kappa^{\tau-}}(V^{-,\tau}) \right) \\
&=& \left(\boxtimes _{\tau \in \arch} \left((V^{+,\tau})^{\otimes d_\kappa^{\tau+}} \boxtimes  (V^{-,\tau})^{\otimes d_\kappa^{\tau-}} \right) \right) c_\kappa \subset V^{\otimes d_\kappa}, 
\end{eqnarray*}
	 where $d_\kappa=\sum\limits_{\tau \in \Sigma}\left(d_\kappa^{\tau+} + d_\kappa^{\tau-}\right)$ and $c_\kappa=\otimes_{\tau \in \Sigma}\left(c_\kappa^{\tau+} \otimes {c}_\kappa^{\tau-}\right)$. We call $c_\kappa$ the {\em generalized Young symmetrizer}.
	 }

We now consider the case of non-positive dominant weight.
If $\kappa^{\tau \pm}$ is a dominant weight, but $\kappa_{\siga\tau}^\tau<0$ we define \[\bS_{\kappa^{\tau \pm}}(V):=\bS_{\kappa^{\tau \pm}-(\kappa_{\siga\tau}^\tau, \hdots, \kappa_{\siga\tau}^\tau)}(V) \otimes \det(V)^{\kappa^\tau_{\siga\tau}}.\] Similarly, if $\kappa^{\tau \pm}$ is dominant, but  $\kappa_{n}^\tau<0$, we define
$\bS_{\kappa^{\tau \pm}}(V):=\bS_{\kappa^{\tau \pm}-(\kappa_n^\tau, \hdots, \kappa^\tau_n)}(V) \otimes \det(V)^{\kappa^\tau_n}$. This allows us to extend the definition of the Schur functor $\bS_\kappa$ to all dominant weights $\kappa$. 

Throughout this paper,  for each dominant weight $\kappa$, we denote the irreducible representation $\bS_\kappa\left(\bigoplus_{\tau \in \Sigma}\left(\bZ_p^{\siga_\tau} \oplus \bZ_p^{\sigb_\tau}\right)\right)$ of $H_{\bZ_p}$ of highest weight $\kappa$ by $\rho_\kappa$. 
In the following, we sometimes write $(\cdot)^{\rho_\kappa}$ for  $\bS_\kappa(\cdot)$, and also $\cE_{\cpct,\kappa}$ (resp. $\cE_{\kappa}$) in place of $\cE_{\cpct,{\rho{_{\kappa}}}}$ (resp. $\cE_{\rho_{{\kappa}}}$).

\begin{rmk}\label{ev-def}
In \cite{CEFMV} (see \cite[Remark~3.5]{CEFMV}) the symbol $\rho_\kappa$ is used for the  representation
	$\wt \rho_\kappa := \Ind_{B^-}^{H_{\bZ_p}} (-\kappa)$
 of $H_{\bZ_p}$ of highest weight $\kappa$, where $B^-$ is the Borel subgroup containing $T$ corresponding to lower triangular matrices in $H_{\bZ_p}$. For each dominant weight $\kappa$, the two representations $\rho_\kappa$ and $\wt\rho_\kappa$ are isomorphic over $\bQ_p$, and by Frobenius reciprocity
 $$ \Hom_{H_{\bZ_p}}(\rho_\kappa, \wt \rho_\kappa)=\Hom_{H_{\bZ_p}}\left(\rho_\kappa, \Ind_{B^-}^{H_{\bZ_p}} (-\kappa)\right) \simeq \Hom_{B^-}\left(\rho_\kappa, \kappa\right) {(\simeq \bZ_p )}.$$
The isomorphism between the left and right hand side is given by composition with $\ev_{\kappa}$, where $\ev_\kappa$ is defined by $\ev_\kappa(f)=f(1)$ for $f \in \Ind_{B^-}^{H_{\bZ_p}} (-\kappa)$.
 In particular, a choice  $\{\lcan^\kappa\}$ of a $\bZ_p$-basis of $\Hom_{B^-}\left(\rho_\kappa, \kappa\right)$ yields an injection $i_\kappa$ from $\rho_\kappa$ into $\wt \rho_\kappa$ such that $\lcan^\kappa=\ev_{\kappa} \circ i_\kappa$. (Note that $\ev_\kappa$ is denoted $\lcan$ in \cite{CEFMV}.) 
\end{rmk}

\subsubsection{Projection onto highest weight representations}\label{defpi}\label{section-Schur2} We will use the material from this section to construct and study differential operators on $p$-adic automorphic forms.  For comparison, we note that a discussion of differential operators on $\ci$ automorphic forms and the description of highest weights in that case is in \cite[Section 12.1]{Shimura} and \cite{shclassical}; a related (but briefer) description also is available in \cite[Section 23]{sh}.
In this section, we denote the standard representation $\bigoplus_{\tau \in \Sigma}\left(\bZ_p^{\siga_\tau} \oplus \bZ_p^{\sigb_\tau}\right)$
 of $H_{\bZ_p}$ by $V$. 
 
Let $\fB= \cup_{\tau} 
\{b_{\tau, 1}, \hdots, b_{\tau, n} \}$ be the standard basis of $V$, 
and $\fB^\vee= \cup_{\tau} \{b_{\tau, 1}^\vee, \hdots, b_{\tau, n}^\vee \}$ be the corresponding dual basis. 
For each positive dominant weight $\kappa$, we write
 \[\pi_{\kappa}:V^{\otimes d_{\kappa}} \twoheadrightarrow \rho_\kappa\] 
for the surjection obtained from projecting onto summands and applying the generalized Young symmetrizer.

\begin{defn}\label{lcan-def}
For each positive dominant weight $\kappa$, we define $\lcan^\kappa$ to be the $\bZ_p$-basis of $\Hom_{B^-}\left(\rho_\kappa, \kappa\right)$  such that 
 \begin{equation} \label{eqn-lcan-formula}
 	 {\tilde{\ell}_{\rm can}^{\kappa} := \lcan^\kappa \circ \pi_\kappa =}  \prod_{\tau \in \Sigma} \prod_{i=1}^n (\kappa_i^\tau !)^{-1} \cdot \bigotimes_{\tau \in \Sigma} \bigotimes_{i=1}^{n} (b_{\tau,i}^\vee)^{\otimes \kappa_i^\tau} \cdot c_\kappa.
 \end{equation}
 \end{defn}

We have chosen the above normalization so that Remark \ref{aipm} holds.

{\begin{defn} \label{sum-sym-def}
A weight $\kappa=(\kappa_1^\tau, \hdots, \kappa_n^\tau)_{\tau \in \Sigma}$ is called {\em sum-symmetric} if $\kappa$ is positive dominant and $d_\kappa^{\tau+}=d_\kappa^{\tau-}$ for all $\tau \in \Sigma$, where $d_\kappa^{\tau+}=\sum_{i=1}^{\siga_\tau}\kappa_i^\tau$ and $d_\kappa^{\tau-}=\sum_{i=\siga_\tau+1}^{n}\kappa_i^\tau$. A representation $\rho$ is called {\em sum-symmetric} if it is isomorphic to $\rho_\kappa$ for some sum-symmetric weight $\kappa$.
		In this case we call $e_\kappa=\sum_{\tau \in \Sigma} d_\kappa^{\tau+} = d_\kappa/2$ the {\em depth} of $\kappa$ or of the representation $\rho_\kappa$.
		\end{defn}
		\begin{defn}\label{sym-def}
		A weight $\kappa$ is called {\em symmetric} if $\kappa$ is sum-symmetric 
		and for all $\tau \in \Sigma$ we have $$\kappa_i^\tau=\kappa_{\siga_\tau+i}^\tau \text{ for all } 1 \leq i \leq \min(\siga_\tau, \sigb_\tau).$$ 
\end{defn} 
\begin{rmk}\label{sum-sym}
		If $\kappa$ is sum-symmetric of depth $e_\kappa$, then, by the Schur functor construction, the representation $\rho_{\kappa}$ of $H_{\bZ_p}$ of highest weight $\kappa$ is a {quotient} of $(\bigoplus_{\tau \in \Sigma}\left(\bZ_p^{\siga_\tau} \otimes \bZ_p^{\sigb_\tau}\right))^{\otimes e_\kappa}$. 
\end{rmk}}

\begin{lem} \label{Lemma-lcan-composition}\label{lemma-Schur2}
	Let  $\kappa$ be a positive dominant weight, let $\kappa'$ be a sum-symmetric weight. Then the projection  $\pi_{\kappa\kappa'}: V^{\otimes d_{\kappa\kappa'}} \twoheadrightarrow {\rho_{\kappa\kappa'}}$ factors through the map
	$$ \pi_{\kappa} \otimes \pi_{\kappa'}: V^{\otimes d_{\kappa\kappa'}}  \simeq V^{\otimes d_{\kappa}} \otimes V^{\otimes d_{\kappa'}} \twoheadrightarrow {\rho_{\kappa}} \otimes {\rho_{\kappa'}}. $$
	Moreover, if we denote the resulting projection $ {\rho_{\kappa}} \otimes {\rho_{\kappa'}} \twoheadrightarrow \rho_{\kappa\kappa'}$ by $\pi_{\kappa,\kappa'}$, then 
	\begin{equation} \label{eqn-lcan-composition}
		\lcan^\kappa \otimes \lcan^{\kappa'} = \lcan^{\kappa\kappa'} \circ \pi_{\kappa,\kappa'},
		\end{equation}
and \begin{equation} \label{eqn-tilde-lcan-composition}
		{\tilde{\ell}}_{\rm can}^\kappa \otimes {\tilde{\ell}}_{\rm can}^{\kappa'} = {\tilde{\ell}}_{\rm can}^{\kappa\kappa'}.
		\end{equation}
\end{lem}
\begin{proof}
Recall the injection $i_\kappa: {\rho_\kappa} \ra \wt \rho_\kappa=\Ind_{B^-}^{H_{\bZ_p}}(-\kappa)$ defined in Remark \ref{ev-def}, and note that $(i_\kappa)_{\bQ_p}$ is an isomorphism. Let $\wt \pi_{\kappa, \kappa'}: \wt \rho_{\kappa} \otimes \wt \rho_{\kappa'} \ra \wt \rho_{\kappa\kappa'}$ be the projection obtained by $f \otimes f' \mapsto ff'$, and define $\pi_{\kappa, \kappa'}:{\rho_\kappa} \otimes {\rho_{\kappa'}} \ra {\rho_{\kappa\kappa'}}$ to be the composition $$(i_{\kappa\kappa'})_{\bQ_p}^{-1} \circ \wt \pi_{\kappa, \kappa'} \circ ((i_{\kappa})_{\bQ_p} \otimes (i_{\kappa'})_{\bQ_p}).$$ 
Then we obtain after base change to $\bQ_p$ that 
$$\lcan^\kappa \otimes \lcan^{\kappa'}= \ev_\kappa \circ i_{\kappa} \otimes \ev_{\kappa'} \circ i_{\kappa'} = \ev_{\kappa\kappa'} \circ \wt \pi_{\kappa,\kappa'} \circ (i_{\kappa} \otimes i_{\kappa'}) =\lcan^{\kappa\kappa'} \circ \pi_{\kappa,\kappa'}.$$ 
Using Equation \eqref{eqn-lcan-formula} and the definition of the action of $c_\kappa$, we deduce that $$\lcan^{\kappa\kappa'} \circ \pi_{\kappa, \kappa'} \circ (\pi_{\kappa}\otimes\pi_{\kappa'}) = \lcan \circ \pi_{\kappa\kappa'}. $$
Thus by Frobenius reciprocity $\pi_{\kappa, \kappa'} \circ (\pi_{\kappa}\otimes \pi_{\kappa'}) = \pi_{\kappa\kappa'}$, and it only remains to check that $\pi_{\kappa,\kappa'}$ is defined over $\bZ_p$. However, this follows from $\pi_{\kappa\kappa'}$ being defined over $\bZ_p$ and the surjectivity of $\pi_{\kappa} \otimes \pi_{\kappa'}$.
\end{proof}

\subsection{The Igusa tower over the ordinary locus}\label{Igusalevel} 

In this section, we introduce the Igusa tower as a tower of finite \'etale Galois covers of the ordinary locus of a Shimura variety. This construction is due to Hida in \cite[Section 8.1]{hida} (see also \cite[Section 4.1]{CEFMV}).
{We recall that our Shimura varieties have hyperspecial level at $p$ and neat level away from $p$  (and that we suppressed the level from the notation).}

We fix a place $P$ of $E$ above $p$ and denote the residue field of $\cO_{E_P} \subset E_P$ by $k$.
Abusing notation, we will still denote the base change of $\cM$ to $\cO_{E_P}$ by $\cM$.  

Let  $\overline{\cM}^{\ord}$ over $k$ be the ordinary locus of $\overline{\cM} = \cM \otimes_{\cO_{E_P}} k$, and $\cM^{\ord}$ over $\cO_{E_P}$ be the ordinary locus of $\cM$ as defined in \cite[Definition 3.4.1.1]{lan4}. Then $\overline{\cM}^{\ord}= {\cM}^{\ord}\otimes_{\cO_{E_P}} k$.
{For each $m\geq 1$, the scheme $\Mord\otimes_{\cO_{E_P}} \Witt/{p}^m \Witt$ agrees with the locus where a lift of (a sufficiently large power of) the Hasse invariant does not vanish. }
Since we assume $p$ splits completely in $K$ (which implies that $p$ splits in the reflex field $E$), $\Mord$ is nonempty, in fact it is open and dense. 
 We fix  a connected component $\Sord$  of ${\Mord_\Witt :=} \Mord\times_{\cO_{E_P}}\Witt$. 
Equivalently, $\Sord$ is the ordinary locus of a fixed connected component $\mathcal{S}$ of $\mathcal{M}_\Witt$.

Let $\cA^{\ord}:=\cA_{\univ/\Sord}$ be the universal (ordinary) abelian variety over $\Sord$. Pick a $\bW$-point $x$ of $\Sord$, and  denote by $\bar x$ the underlying $\overline{\bF}_p$-point.  We can identify the $\zz_p$-lattice $L_p$ (defined in Section \ref{PEL-data}) 
with the $p$-adic Tate module of $\cA^{\ord}_x[p^\infty]$.  Choose such an identification $L_p\simeq T_p(\cA^{\ord}_x[p^\infty])$, compatible with the $\cO_K$-action and identifying the Hermitian pairing with the Weil pairing.
Then, the kernel of the reduction map 
	$$T_p(\cA^{\ord}_x[p^\infty]) \to T_p(\cA^{\ord}_{\bar{x}}[p^\infty]^{\et})$$ 
determines an $\cO_K$-submodule $\cL \subset L_p$.  Using the self-duality of $L_p$ under the Hermitian pairing $\langle\cdot,\cdot\rangle$ and its compatibility with the $\lambda$-Weil-pairing $(\cdot,\cdot)_{\lambda}$, we can identify the dual $\cL^\vee$ of $\cL$ with the orthogonal complement of $\cL$ inside $L_p$.   Note that $\cL$ decomposes as 
	$$\cL=\oplus_{\tau \in \Sigma}(\cL_\tau^+\oplus\cL_\tau^-).$$
In the sequel, we write
	\begin{equation}\label{Lsquared}
	\cL^2:=\oplus_{\tau \in \Sigma}\cL^+_\tau\otimes\cL_\tau^-.
	\end{equation}

We now introduce the {\it Igusa tower} over the $p$-adic completion of $\Sord$. 
For each $m\in\mathbb{Z}_{\geq 1}$, we write $\Sordm := \Sord\times_{\Witt}\Witt/p^m\Witt$. For each $n,m\in\mathbb{Z}_{\geq 1}$, consider the functor 
\begin{align*}
\mathrm{Ig}_{n,m}: \left\{\mathrm{Schemes}/\Sordm\right\}\rightarrow\left\{\mathrm{Sets}\right\}
\end{align*}
that takes an $\Sordm$-scheme $S$ to the set of $\OK$-linear closed immersions \[\iota_n:\cL \otimes_{\mathbb{Z}_p}\mu_{p^n} \hookrightarrow \cA_S[p^n],\] where $\cA_S:=\cA^{\ord}\times_{\Sord}S$. This functor is represented by an $\Sordm$-scheme, which {by abuse of notation} we also denote by $\mathrm{Ig}_{n,m}$. For each $n\geq 1$, $\mathrm{Ig}_{n,m}$ is a finite \'etale and Galois cover of $\Sordm$, whose Galois group is the group of $\cO_K$-linear automorphisms of $\cL/p^n\cL$.

For each $n\geq 1$, we define the formal scheme $\Ig_n:=\varinjlim_m \Ig_{n,m}$. Equivalently, we define $\Ig_n$ as the formal completion along the special fiber of the scheme representing the functor that takes a $\Sord$-schemes $S$ to the set of $\cO_K$-linear closed immersions $\iota_{n}: \cL \otimes_{\bZ} \mu_{p^n} \hookrightarrow \cA_S[p^n]$. 

Finally, we define the {\em infinite Igusa tower} $\Ig$ as $\Ig:=\varprojlim_n \Ig_{n}$.  Recall that inverse limit of projective system of formal schemes, with affine transition maps, exists in the category of formal schemes (see \cite[Proposition D.4.1]{fargues}.)  Thus, $\Ig$ exists as formal scheme, and is a pro-\'etale cover of the formal completion of $\Sord$ along its special fiber, with Galois group the group of $\cO_K$-linear automorphisms of $\cL$, which we identify with $H(\bZ_p)$.
For any point $x_0$ of the formal completion of $\Sord$ along its special fiber (e.g., $x_0\in\Sord(\Witt)$ or $\Sord(\overline{\bF}_p)$), 
the choice of a point $x$ of $\Ig$ lying above $x_0$ is equivalent to the choice of an Igusa structure of infinite level on ${\cA}_{x_0}$, i.e. of an $\cO_K$-linear closed immersion of Barsotti-Tate groups
$\iota_x:\cL \otimes_{\bZ_p} \mu_{p^\infty}\hookrightarrow {\cA}_{x_0}[p^\infty]$.
In the following, we write
	$$\iota:\cL \otimes_{\bZ_p} \mu_{p^\infty}\hookrightarrow {\cA}^{\ord}[p^\infty]$$ for the universal Igusa structure of infinite level on $\cA^{\ord}$ over $\Igusa$.

\subsection{$p$-adic automorphic forms}\label{padicaut-section}\label{PSI}
Following Hida \cite[Section 8.1]{hida}, we define $p$-adic automorphic forms as global functions on the Igusa tower (see also \cite[Section 4.2]{CEFMV}). 
For all $n,m\in\bZ_{\geq1}$, let
	$$V_{n,m} := H^0(\Ig_{n,m},\cO_{\Ig_{n,m}}),$$
we write $V_{\infty,m}:=\varinjlim_n V_{n,m}$, and $V_{\infty,\infty}:=\varprojlim_m V_{\infty,m}$. 

Note that the space $V_{\infty,\infty}$ is endowed with a left action of $\Levin(\zz_p)$, $f\mapsto g\cdot f$, induced by the natural right action of $g \in \Levin (\zz_p)$ on the Igusa tower.

We call 
$V^N :=V_{\infty,\infty}^{N(\bZ_p)}$ the {\it space of $p$-adic automorphic forms}.

The above definition is motivated by the existence of an embedding of the space of $p$-adic automorphic forms, regarded as global sections of automorphic vector bundles on $\Sord$, into $V^N$. We briefly recall the construction ({\cite[Section 8.1.2]{hida})  adapted to our setting.

Fix $n\geq m \geq 0$, and $\kappa$ any dominant weight.  Let  $V_{n,m}^N[\kappa]$ denote the $\kappa$-eigenspace of the action of torus on $V_{n,m}^{N(\bZ_p)}$.
We define a map 
$$\Psi_{n,m}^\kappa: H^0\left({\Sordm}, \Ekap\right) \rightarrow V_{n,m}^N[\kappa]$$ as follows.
We regard each $f\in H^0\left({\Sordm}, \Ekap\right)$  as a function 
$\left(\uA, j \right)\mapsto f\left(\uA, j\right)\in \rho_\kappa(\Witt_m)$
on pairs  $(\uA,j)$, where $\uA=\underline{\cA}^\ord_{x_0}$ is an abelian variety associated to a point $x_0$ of $\Sordm$, and $j$ is a
the trivialization of  $\uo_{\uA}=\uo_{\underline{\cA}^{\ord}/\cS^\ord_m, x_0}$.
Using the canonical isomorphism
	$$\uo_{\cA^{\ord}/\cS^{\ord}_m} \cong \cA^{\ord}[p^n]^{\et} \otimes \cO_{\cS^{\ord}_m},$$
to each Igusa structure $\iota$ on $\uA$ we associate a trivialization $j_\iota$ of $\uo_{\uA}$. Finally, we define 
$ \Psi_{n,m}({f})\in V_{n,m}^N[\kappa]$ as the function $(\uA, \iota) \mapsto \ell_{\rm can}^\kappa(f(\uA, j_\iota))$.
As $n,m$ vary, with $n\geq m$, we obtain a map
\begin{align}\label{psik-map}
\Psi_\kappa:H^0\left({\Sord}, \Ekap\right)\rightarrow V^N[\kappa],
\end{align}
where $V^N[\kappa]$ denotes the $\kappa$-eigenspace of the action of the torus on $V^N$.

We define
\begin{align}
\Psi:\bigoplus_{\kappa \in X(T)_+} H^0(\Sord,\Ekap) \ra V^N
\end{align}
to be the {linear} map whose restriction to $H^0(\Sord,\Ekap)$ is $\Psi_\kappa$.

{\begin{thm}{\cite[Prop.~8.2 \& Thm.~8.3]{hida}}\label{density}
	The map $\Psi_\kappa$ is injective, and after inverting $p$, the image of $\Psi$ 
	$$\Psi\left( \bigoplus_{\kappa \in X(T)_+, \atop \kappa \text{ positive }} H^0(\Sord,\Ekap)\right)\left[\frac{1}{p}\right] \cap V^N$$
	is $p$-adically dense in $V^N$. 	
\end{thm} }{
\begin{proof} 
Proposition 8.2 and Theorem 8.3 in \cite{hida} are conditional on the assumption (given in \cite[Section 8.1.4]{hida}) that the following equality holds for all $\kappa \in X(T)_+$ and all integers $m \geq 1:$
\begin{equation*}H^0(\Sord ,\cE_\kappa)/p^m H^0(\Sord ,\cE_\kappa)= H^0(\Sordm ,\cE_{\kappa}).\end{equation*} Here, we prove that such equations hold in our settings. 

Although they have not been introduced in this paper, Lan has constructed partial toroidal and minimal compactifications $\Sord^{\mathrm{tor}}$ and $\Sord^{\mathrm{min}}$ of $\Sord$ (see \cite[Theorems 5.2.1.1 \& 6.2.1.1]{lan4}) as well as a canonical extension $\Ekap^{\mathrm{can}}$ of $\Ekap$ to the partial toroidal compactifications (see \cite[Definition 8.3.3.1]{lan4}). Additionally, by \cite[Proposition 6.3.2.4]{lan4}, for every $m \geq 1$, we have that $\Sord^{\rm{min}} \times_{\bW} \bW/p^m\bW$ is affine. Because the pushforwards (under a proper map by \cite[Proposition 5.2.3.18]{lan4}) of $\Ekap^{\mathrm{can}}$ to $\Sord^{\rm{min}}$ are quasi-coherent, we can conclude that
	$$H^0\left(\Sord^{\mathrm{tor}},\Ekap^{\mathrm{can}}\right)/p^m H^0\left(\Sord^{\mathrm{tor}},\Ekap^{\mathrm{can}}\right) = H^0\left(\Sord^{\mathrm{tor}} \times_{\bW} \bW/p^{m}\bW,\Ekap^{\mathrm{can}}\right).$$
We could then conclude the theorem if we knew that Koecher's Principle applied. By \cite[Remark 10.2]{lan5}, the analogue of \cite[Theorem 2.3]{lan5} holds for the partial compactifications of ordinary loci and so we deduce 
	$$H^0(\Sord^{\mathrm{tor}},\Ekap^{\mathrm{can}}) = H^0(\Sord ,\cE_\kappa).$$
\end{proof}}

The above statement implies that the $p$-adic closure of the space of integral weight $p$-adic automorphic forms is the space of all $p$-adic automorphic forms.

\subsection{Serre--Tate theory for unitary Shimura varieties}
We briefly recall the main results in \cite{CEFMV}. 

\subsubsection{Local coordinates at ordinary points}\label{ST-coord}

For any $\xzero\in \Sord(\Witt)$, we write $\xzeroreduced\in \Sord(\fpb)$ for its reduction modulo $p$, and   
denote  by $\Sord^\wedge_{\xzeroreduced}$ the formal completion of $\Sord \times \overline{\bF}_p$ at $\xzeroreduced$. We also write $\Sord_{\xzeroreduced}^\wedge={\rm Spf}(\Ring)$. The ring $\Ring$ is a complete local ring over $\Witt$, with residue field $\fpb$, and we denote by $\fm_{\xzeroreduced}$ its maximal ideal.

 Let ${A_0}=\cA^{\ord}_{\xzeroreduced}$ denote the abelian variety {over $\ov \bF_p$} 
attached to the point $\bar{x}_0$, and write $T_p{A_0}$ for the physical Tate module of ${A_0}$. It is a free $\cO_K\otimes_{\zz}\zz_p$-module, which  decomposes as
\[T_p{A_0}=\oplus_{i=1}^r T_{\fP_i} {A_0}\bigoplus \oplus_{i=1}^r T_{\fP_i^c}{A_0},\]
 where the decomposition is induced from the identification $\cO_K \otimes \bZ_p \cong \bigoplus_{i=1}^r \left( \cO_{K_{\fP_i}} \oplus \cO_{K_{\fP_i^c}} \right)$.

\begin{thm}(\cite[Proposition 5.8]{CEFMV})\label{ST-isom}\label{ST-coord-thm}
Let $\xzero\in\Sord(\Witt)$. There exists a canonical isomorphism of formal schemes 
\[{\Sord}^\wedge_{\bar{x}_0}\isomto \bigoplus_{i=1}^r{\rm Hom}_{\zz_p} (T_{\fP_i}{A_0}\otimes T_{\fP_i^c}{A_0},\hat{\mathbb G}_m), \quad x\mapsto q_x
.\]
\end{thm}

In the following, we identify the space $\bigoplus_{i=1}^r{\rm Hom}_{\zz_p} (T_{\fP_i}{A_0}\otimes T_{\fP_i^c}{A_0},\hat{\mathbb G}_m)$ with the subspace of ${\rm Hom}_{\zz_p} (T_p{A_0}\otimes T_p{A_0},\hat{\mathbb G}_m) $ consisting of all symmetric $(\cO_K\otimes_{\zz}\zz_p, c)$-hermitian forms, 
and write 
\begin{align}\label{fq-equ}
\fq=q_{\cA/\Sord^\wedge_{\bar{x}_0}}:T_p{A_0}\otimes T_p{A_0}\to \hat{\mathbb G}_m 
\end{align} 
for the universal symmetric $(\cO_K\otimes_{\zz}\zz_p, c)$-hermitian form over $\Sord^\wedge_{\bar{x}_0}$. 
This implies that $\fq$ satisfies $\fq(Q,P)=\fq(P,Q)$ and $\fq(kQ,P) = \fq(Q,k^cP)$,  for all $P,Q\in T_pA_0$ and $k \in \cO_K$. In particular, for any $P\in T_{\fP_i}A_0$, $\fq(Q,P)=0$  unless $Q\in T_{\fP_i^c}A_0$.

For any point $x\in\Igusa(\Witt)$ above $\xzero$, we write $\bar{x}$ for  its reduction modulo $p$, and $\iota_{\bar{x}}$ for the Igusa structure {of infinite level} on $\underline{{A_0}}$ attached to the point $\bar{x}$. 
The map $\iota_{\bar{x}}:\cL \otimes \mu_{p^\infty}\hookrightarrow {A_0}[p^\infty]$ induces an  isomorphism of $\cO_K\otimes_{\zz}\zz_p$-modules 
\begin{equation}\label{TateLattice}
T_p(\iota_{\bar{x}}^\vee):T_p{A_0}\to \cL^\vee .
\end{equation} 
We denote by \[t_x:\bigoplus_{i=1}^r T_{\fP_i}{A_0}\otimes T_{\fP_i^c}{A_0}\isomto (\cL^2)^\vee,\]
the $\zz_p$-linear isomorphism induced by the restriction of $T_p(\iota_{\bar{x}}{\dual})^{\otimes 2}$.

\begin{prop}(\cite[Proposition {5.10}]{CEFMV})\label{ST-coordinates}\label{beta}
Let $\xzero\in\Sord(\Witt)$.  Each point $x\in\Igusa(\Witt)$ above $\xzero$ defines an isomorphism of formal schemes $\beta_x: {\Sord}^\wedge_{\bar{x}_0}\isomto \gm\otimes \cL^2$. 
\end{prop}

\begin{rmk}\label{beta-rmk}
Let $t$ denote the canonical formal parameter on  $\gm$, we write 
\[\beta_x^*:  \Witt[[t]]\otimes (\cL^2)^\vee \isomto \Ring\]
for the isomorphism of local rings induced by $\beta_x${, where $\Witt[[t]]\otimes(\cL^2)^\vee$ denotes the complete ring corresponding to the formal scheme ${\hat{\bG}}_m\otimes \cL^2$.  A choice of a $\zz_p$-basis $\fE$ of $(\cL^2)^\vee$ yields} the isomorphism
\[\beta^*_{x,\fE}:\Witt[[t_l|l\in {\fE} ]]\isomto \Ring , \]
{which} satisfies the {equality} $\beta_x^*(t_l)=\fq(t_x^{-1}(l)) -1\in \fm_{\bar{x}_0}$, for all $l\in\fE$.
\end{rmk}

\subsubsection{The $t$-expansion principle} \label{section-t-expansion}

Let $x\in\Igusa(\Witt)$. Recall that since $\Igusa$ is a pro-finite \'etale cover of $\Sord$, the natural projection $j:\Igusa\to \Sord$ induces an isomorphism between the formal completion of $\Igusa$ at $x$ and $\Sord^\wedge_{\bar{x}_0}$, for $\xzero=j(x)\in \Sord(\Witt)$. In particular,   the localization map at $x$  induces a map ${\rm loc}_x:V_{\infty,\infty}\rightarrow \Ring$. 

 For any $f\in V^N$, {\em the $t$-expansion} (or {\em Serre--Tate expansion}) of $f$ is defined as \[f_x(t):={\beta^*_x}^{-1}({\rm loc}_x(f))\in W[[t]]\otimes (\cL^2)^\vee.\]
Recall that, for all $g\in H(\bZ_p)$, we have $f_{x^g}(t)=({\rm id}\otimes g^{-1})(g\cdot f)_x(t)$ (\cite[Proposition 5.13]{CEFMV}).

\begin{thm}(\cite[Theorem 5.14, Proposition 5.5, Corollary 5.16]
{CEFMV})\label{t-exp}
\begin{enumerate}
\item For any weight $\kappa$, and $f\in V^N[\kappa]$: $f_x(t)=0$ if and only if $f=0$.
\item For any $f\in V^N$, $f=0$ if and only if $f_x(t) = 0$ for at least one CM point $x$ in each connected component of the Igusa tower.  In particular, for a choice of a CM point $x$, $f=0$ if and only if $(g\cdot f)_x(t)=0$  (equiv. $f_{x^g}(t)=0$) for all $g\in T(\zz_p)$. 
\item
 Let $m\in\bN$. {Let} $f,f'\in V^N$ be two $p$-adic automorphic forms of weight $\kappa$ and $\kappa'$, respectively. {Then}
$f\equiv f' \mod p^m$ if and only if  for all $g\in T(\zz_p)$ \[\kappa(g) f_x(t)\equiv \kappa'(g)f'_x(t)\mod p^m.\]
\end{enumerate}
 \end{thm}

\section{Differential operators}\label{diffop-sec}
In this section, we introduce differential operators similar to the ones in \cite{E09, EDiffOps}.  Unlike \cite{E09, EDiffOps} (which only explicitly handles unitary groups whose signature is of the form $\left(\siga, \sigb\right)$ with $\siga=\sigb$ at each archimedean place), we place no restrictions on the signature of the unitary groups with which we work.

\subsection{The Gauss--Manin connection}\label{GM-section}
We briefly review key features of the Gauss--Manin connection, which was first introduced by Y. Manin in \cite{manin} and later extended and studied by N. Katz and T. Oda \cite{katzGM, KO}.  A detailed summary of the Gauss--Manin connection also appears in \cite[Section 3.1]{EDiffOps}.  Below, we mostly follow the approaches of \cite[Section 2]{KO} and \cite[Section 3.1]{EDiffOps}. 
Throughout this section, let $S$ be a smooth scheme over a scheme $T$, and let $\pi: X\rightarrow S$ be a smooth proper morphism of schemes. {Define $\hdr^q(X/S)$ to be the relative de Rham sheaf in the complex $\hdr^\bullet(X/S)$, i.e. the quasi-coherent sheaf of graded algebras on $S$ given by
\begin{align*}
\hdr^q(X/S):= \R^q\pi_*\left(\Omega_{X/S}^\bullet\right),
\end{align*}
here $\R^q\pi_*$ denotes the $q$-th hyper-derived functor of $\pi_*$, and $\Omega_{X/S}^\bullet$ denotes the complex {$\bigwedge^\bullet \Omega^1_{X/S}$} on $X$ whose differentials are induced by the canonical K{ä}hler differential $\cO_{X/S} \rightarrow \Omega_{X/S}^1$.} 
 The de Rham complex $\left(\Omega_{X/T}^\bullet, d\right)$ admits a canonical filtration
\begin{align*}
\Fil^i\left(\Omega^\bullet_{X/T}\right):=\Im\left(\pi^*\Omega^i_{S/T}\otimes_{\cO_X}\Omega_{X/T}^{\bullet-i}\rightarrow\Omega_{X/T}^\bullet\right),
\end{align*}  
 with associated graded objects 
\begin{align*}
\Gr^i(\Omega^{\bullet}_{X/T})\cong \pi^*\Omega^i_{S/T}\otimes_{\cO_X} \Omega^{\bullet-i}_{X/S}
\end{align*} 
(this follows from the exactness of the sequence	$0 \rightarrow \pi^{\ast}\Omega^1_{S/T} \rightarrow \Omega^1_{X/T} \rightarrow \Omega^1_{X/S}\rightarrow 0$ for $\pi$ smooth).
 Using the above filtration, one obtains a spectral sequence $(E^{p,q}_r)$ converging to $\bR^q \pi_{\ast}(\Omega^\bullet_{X/T})$, whose first page is
\begin{align}\label{e1pq}
E_1^{p, q} = \R^{p+q}\pi_*\left(\Gr^p\right)\cong \Omega^p_{S/T}\otimes_{\cO_S}\hdr^q\left(X/S\right)
\end{align}
and such that the {\em Gauss--Manin connection} $\nabla$ is the map
$d_1^{0, q}: E_1^{0, q}\rightarrow E_1^{1, q}.$
 Using Equation \eqref{e1pq}, we regard $\nabla$ as a map
\begin{align*}
\nabla: \hdr^q\left(X/S\right)\rightarrow \hdr^q\left(X/S\right)\otimes_{\cO_S}\Omega^1_{S/T}.
\end{align*}
 It is an integrable connection.  In this paper, we shall be interested solely in the case of $q=1$.

\subsection{The Kodaira--Spencer morphism}\label{KS-section}

We now briefly review the construction of the Kodaira--Spencer morphism, focusing on the details we need for this paper.  More detailed treatments than we shall need for the present paper are available in \cite[Section 3.2]{EDiffOps}, \cite[Sections 2.1.6-7 \& 2.3.5]{lan}, and \cite{CF, EDiffOps}.  Like in Section \ref{GM-section}, we let $S$ be a smooth scheme over a scheme $T$, and we let $\pi: A\rightarrow S$ be a smooth proper morphism of schemes, and we require $A$ to be an abelian scheme together with a polarization $\polarization:A\rightarrow A\dual$. We define\footnote{In \cite{lan,lan4}, Lan gives an equivalent definition for $\uo_{A/S}$ as $  e^{\ast}\Omega^1_{A/S}$, the pullback via the identity section of the sheaf of relative differentials on $A$.}  $\uo_{A/S} := \pi_{\ast}\Omega^{{1}}_{A/S}$. 
The Kodaira--Spencer morphism is a morphism of sheaves
$$
\KS: \uo_{A/S}\otimes\uo_{{A\dual/S}}\twoheadrightarrow \Omega^1_{S/T},
$$
defined as follows. {Consider the exact sequence 
\begin{align}\label{uoincl-maps}
0\rightarrow \uo_{X/S}\hookrightarrow\hdr^1(X/S)\twoheadrightarrow H^1(X,\cO_X)\rightarrow 0
\end{align}
obtained by taking the first hypercohomology of the exact squence $0 \rightarrow \Omega^{\bullet\geq1}_{X/S} \rightarrow \Omega^{\bullet}_{X/S} \rightarrow \cO_X \rightarrow 0$ where we view $\cO_X$ as a complex concentrated in degree 0. 
By identifying  $H^1(A,\cO_A) \cong \uo_{A^\vee/S}^{\vee}$,  we obtain:
\begin{align}\label{xdualsurj-map}
0\rightarrow\uo_{A/S}\hookrightarrow\hdr^1(A/S)\twoheadrightarrow \uo_{A\dual/S}\dual\rightarrow 0.
\end{align}
The {\em Kodaira--Spencer morphism} $\KS$ is defined to be the composition of morphisms: }
\begin{equation*} 
\xymatrixrowsep{.3in}
\xymatrixcolsep{.17in}
 \xymatrix{ 
 \hdr^1\left(A/S\right) \otimes\uo_{A\dual/S}\ar[r]^{[2]\qquad} &
\hdr^1\left(A/S\right)\otimes\Omega^1_{S/T}\otimes\uo_{A\dual/S} \ar@{->>}[r]^{\quad [3]}   &
{\uo_{A\dual/S}\dual\otimes \Omega^1_{S/T} \otimes\uo_{A\dual/S}} \ar@{->>}[d]^{[4]}  \\ 
\uo_{A/S}\otimes\uo_{A\dual/S}\ar@{^{(}->}[u]^{[1]} \ar@{-->>}[rr]^{KS} & &
\Omega^1_{S/T} }\end{equation*}	
where [1] is the canonical inclusion from \eqref{xdualsurj-map} tensored with the identity map on $\uo_{A\dual/S}$, [2] is $\nabla\otimes \id_{\uo_{A\dual/S}}$, [3] is the surjection in \eqref{xdualsurj-map} tensored with $\id_{\Omega^1_{S/T}} \otimes \id_{\uo_{A\dual/S}}$, and [4] is the pairing $\uo_{A\dual/S}\dual\otimes \uo_{A\dual/S}\rightarrow\cO_S$ tensored with   $\id_{\Omega^1_{S/T}}$.

By identifying $\uo_{A/S}$ with $\uo_{A\dual/S}$ via the polarization $\polarization:A\rightarrow A\dual$, we regard $\KS$ as a morphism
\begin{align}\label{KS}
\KS: \uo_{A/S}\otimes_{\cO_S}\uo_{A/S}\twoheadrightarrow\Omega^1_{S/T}.
\end{align}

We now assume  $S$ is a scheme equipped with an \'etale morphism $S \rightarrow \cS^{\ord}_T$ where $T$ is a scheme over $\bW$.{ We write $A$ for the corresponding abelian scheme over $S$; the action of $\cO_K$ on $A$ induces a decomposition
\begin{align}\label{decomptau}
\uo_{A/S} = \bigoplus_{\tau \in \arch} (\uo_{{A/S},\tau}^+ \oplus \uo_{{A/S},\tau}^-)
\end{align}
defined as in  \eqref{decomptau1}.  In the following, we write
\begin{align}\label{omega2}
\underline{\omega}_{A/S}^2:=\oplus_{\tau\in\arch}\left(\uo_{A/S, \tau}^+\otimes\uo_{A/S, \tau}^-\right).\end{align}}
\begin{prop}\label{etaleks-prop}\label{ks} ({\cite[Proposition 3.4.3.3]{lan4}}) For any \'etale morphism $S\rightarrow \cS^{\ord}_{T}$ over $T$, $\KS$ induces an isomorphism\begin{align}\label{KSiso}{\rm ks}:  \underline{\omega}_{A/S}^2\isomto \Omega^1_{S/T}.
\end{align}
\end{prop}

\subsection{Definitions of differential operators} \label{doperator}\label{diffop}
We now define differential operators.  The construction is the same as the one in \cite[Sections 7-9]{EDiffOps}, which follows the construction in \cite[Chapter II]{kaCM}.  Unlike in \cite{E09, EDiffOps}, we place no conditions on the signature of the unitary groups with which we work; but the construction is identical.  The place the generalization of the signature is apparent is in the explicit description of the operators in terms of coordinates in Section \ref{coord-section}.

Let $A$, $S$, and $T/\bW$ be as in Proposition \ref{etaleks-prop}.   We identify $\Omega^1_{S/T}$ with $\oplus_{\tau\in\arch}\left(\uo_{A/S, \tau}^+\otimes\uo_{A/S, \tau}^-\right)$ via the isomorphism \eqref{KSiso}, and  $\oplus_{\tau\in\arch}\left(\uo_{A/S, \tau}^+\otimes\uo_{A/S, \tau}^-\right)$ with its image in $\hdr^1\left(A/S\right)^{\otimes 2}$ via the inclusion \eqref{uoincl-maps}.
 Applying Leibniz's rule (i.e. the product rule) together with the Gauss--Manin connection $\nabla$, we obtain an operator
\begin{align*}
\nabla_{\otimes d}: \hdr^1\left(A/S\right)^{\otimes d} \rightarrow \hdr^1\left(A/S\right)^{\otimes (d+2)}
\end{align*}
for all positive integers $d$.  

 The $\cO_K \otimes \bW$-structure on $A$ induces a decomposition 

\begin{align*}
\hdr^1\left(A/S\right)
={\bigoplus}_{\tau\in\arch}\left(\hdr^{+,\tau}\left(A/S\right)\oplus \hdr^{-, \tau}\left(A/S\right)\right),
\end{align*}

{such that 
$\uo_{A/S, \tau}^\pm\subset\hdr^{\pm,\tau}\left(A/S\right)$  and  
$\nabla(\hdr^{\pm{, \tau}}\left(A/S\right))\subseteq \hdr^{\pm{, \tau}}\left(A/S\right)\otimes_{\cO_S}\Omega^1_{S/T},$
for all $\tau\in\arch$ (\cite[Equations (3.3)-(3.4)]{EDiffOps}).}
Thus, the image of $\nabla_{\otimes d}$ is contained in $\hdr^1\left(A/S\right)^{\otimes d}\otimes {\left(\bigoplus_{\tau \in \Sigma}\left(\hdr^{+,\tau}\left(A/S\right)\otimes\hdr^{-,\tau}\left(A/S\right)\right)\right)}$.

For all positive integers $d$ and $e$, we define $\nabla_{\otimes d}^e:=\nabla_{\otimes (d+2(e-1))}\circ\nabla_{\otimes \left(d+2\left(e-2\right)\right)}\circ\cdots\circ \nabla_{\otimes d}$,
\begin{align*}
\nabla_{\otimes d}^e: \hdr^1\left(A/S\right)^{\otimes d}\rightarrow \hdr^1\left(A/S\right)^{\otimes d}\otimes\left({\bigoplus_{\tau \in \Sigma}\left(\hdr^{+,\tau}\left(A/S\right)\otimes\hdr^{-,\tau}\left(A/S\right)\right)}\right)^{\otimes e}.
\end{align*}

\begin{prop}\label{schur}
For each positive integer $e$ and each positive dominant weight $\kappa$, the map $\nabla^e_{\otimes_d}$ {where $d=d_\kappa$} {induces a map}
\[\nabla^e_\kappa:  \bS_\kappa\left(\hdr^1\left(A/S\right)\right)\rightarrow \bS_\kappa\left(\hdr^1\left(A/S\right)\right)\otimes\left({\bigoplus_{\tau \in \Sigma}\left(\hdr^{+,\tau}\left(A/S\right)\otimes\hdr^{-,\tau}\left(A/S\right)\right)}\right)^{\otimes e}.\]
\end{prop}
\begin{proof}
Note that by definition the operator $\nabla_{\otimes d}$ is equivariant for the action of $\Symm_d$ (where we consider the natural action on the $d$-th tensor power
and the action on the $d+2$-th tensor power induced by the standard inclusion $\Symm_d\to\Symm_{d+2}$). 
I.e., { \[\nabla_{\otimes d} \left( \left( \cdot \right) \sigma \right) = \left( \nabla_{\otimes d} \left( \cdot \right) \right) \sigma  \text{ for all }\sigma\in\Symm_d,\] }
for all positive integers $d$. It follows from the definition that the same holds for the operators $\nabla_{\otimes d}^e$ for all positive integers $d$ and $e$. Thus, in particular {\[\nabla^e_{\otimes_d}(f \cdot c_\kappa)=(\nabla^e_{\otimes d}f) \cdot c_{\kappa},\]} for all $f\in  \hdr^1\left(A/S\right)^{\otimes d}$, and $c_\kappa$ the generalized Young symmetrizer of $\kappa$.
\end{proof}
For a locally free sheaf of modules $\mathcal{F}$, we sometimes write $\left(\mathcal{F}\right)^\rho$ (resp.   $\nabla_\rho^e$) in place $\bS_\kappa \left(\mathcal{F}\right)$ (resp.  $\nabla_\kappa^e$), for $\rho=\rho_\kappa$ the irreducible representation with highest weight $\kappa$.

Note that {$\nabla_{\rho}:=\nabla_\rho^1$} decomposes a sum  over $\tau\in\arch$ of maps
\begin{align*}
\nabla_{\rho}(\tau):  \left(\hdr^1\left(A/S\right)\right)^\rho\rightarrow \left(\hdr^1\left(A/S\right)\right)^\rho\otimes\left(\hdr^{+,\tau}\left(A/S\right){\otimes} \hdr^{-, \tau}\left(A/S\right)\right).
\end{align*}
 For each nonnegative integer $e$, we define $\nabla_\rho^e(\tau)$ to be the composition of $\nabla(\tau)$ with itself $e$ times (taking into account that the subscript changes with each iteration).

\subsubsection{$\ci$ differential operators}\label{cidefns-section}

The construction of the $\ci$ differential operators in this section is similar to the one in \cite[Section 2.3]{kaCM} and \cite[Section 8]{EDiffOps}.  As explained in \cite[Section 8.3]{EDiffOps}, these differential operators are the Maass--Shimura operators discussed in \cite[Section 12]{Shimura}.  (The explanation in \cite{EDiffOps} immediately extends to all signatures.)

 Let $\hdr^1\left(\ci\right):=\hdr^1\left({\Auniv/\cM}\right)\left(\ci\right)$, and  $\uo\left(\ci\right):=\uo_{\Auniv/\cM}\left(\ci\right)$.
By {Equation \eqref{decomptau}, we have a decomposition
\begin{equation}\label{decomptauinf}
\uo\left(\ci\right) = \bigoplus_{\tau \in \arch} \left(\uo_{\tau}^+\left(\ci\right) \oplus \uo_{\tau}^-\left(\ci\right)\right)
\end{equation}}
{and similarly for $\hdr^1(\ci)$.}

The Hodge decomposition 
$\hdr^1\left(\ci\right)=\uo\left(\ci\right)\oplus\overline{\uo\left(\ci\right)}$ over $\modulispace\left(\ci\right)$ (following the convention in \cite[Section 1.8]{kaCM}, the bar denotes complex conjugation)  induces decompositions \[\hdr^{\pm{,\tau}}\left(\ci\right)=\uo_{{\tau}}^\pm\left(\ci\right)\oplus\overline{\uo_{{\tau}}^\pm\left(\ci\right)},\] 
for all $\tau\in\Sigma$, and the associated projections 
$
\hdr^{\pm{, \tau}}\left(\ci\right)\twoheadrightarrow\uo^{\pm}_{{\tau}}(\ci)
$
induce projections 
\[\varpi_\rho(\ci):\hdr^1\left(\ci\right)^\rho\twoheadrightarrow\uo(\ci)^\rho,\]
for all irreducible representations $\rho$ as above.

As in \cite[(1.8.6)]{kaCM} and \cite[Section 8]{EDiffOps}, 
 $\nabla(\overline{\uo(\ci)})\subseteq\overline{\uo(\ci)}\otimes\Omega_{\modulispace(\ci)/\bC}.$  We define 
\begin{align*}
D_\rho(\ci): \uo(\ci)^\rho\rightarrow\uo(\ci)^\rho\otimes{\bigoplus{_{\tau\in\Sigma}\left(\uo_\tau^+(\ci)\otimes\uo_\tau^-(\ci)\right)}}
\end{align*}
to be the restriction of {$(\varpi_\rho(\ci)\otimes\id)\circ\nabla_\rho$} to $\uo(\ci)^\rho$.

For each irreducible representation $\mathcal{Z}$ of $H$ that is sum-symmetric of some depth $e$, 
let $\pi_{\mathcal{Z}}$ be the
projection of 
${\left(\bigoplus_{\tau\in\Sigma}\left(\uo_\tau^+(\ci)\otimes\uo_\tau^-(\ci)\right)\right)^e}$ onto 
$\uo(\ci)^{\mathcal{Z}}$ defined as in Section \ref{section-Schur2} (i.e. by projection onto summands and applying the generalized Young symmetrizer $c_\cZ$).

We define
\begin{align*}
D_\rho^\mathcal{Z}(\ci) := ({\rm id}\otimes\pi_{\mathcal{Z}})\circ D_{\rho}^e(\ci): \uo(\ci)^\rho \to \uo(\ci)^\rho\otimes \uo(\ci)^\cZ .
\end{align*}
As explained at the end of \cite[Section 8.1]{EDiffOps}, the operators $D_\rho^\mathcal{Z}(\ci)$ canonically induce operators, which we also denote by $D_\rho^\mathcal{Z}$,
 \begin{align*}
D_\rho^\mathcal{Z}(\ci): \cE_{\rho}(\ci)\rightarrow \cE_{\rho\otimes\mathcal{Z}}(\ci).
\end{align*}
(Many additional details of these operators, explicitly for unitary groups of signature $(n,n)$ but which extend by similar arguments to the case of arbitrary signature, which we do not need in the present paper, are discussed in \cite[Section 8]{EDiffOps}.)

Let $\kappa'$ be a sum-symmetric weight of depth $e$, and $\kappa$ be a positive dominant weight.  Take $\rho=\rho_\kappa$ and $\cZ=\rho_{\kappa'}$.  By abuse of notation, we still denote by $\pi_{\kappa, \kappa'}$  the projection $\cE_{\rho\otimes \cZ} \to\cE_{\rho_{\kappa\cdot\kappa'}}$ induced by the projection $\pi_{\kappa, \kappa'}:\rho_\kappa\otimes\rho_{\kappa'}\to \rho_{\kappa\kappa'}$ defined in Lemma \ref{lemma-Schur2}.
We define 
\begin{align*}
D_\kappa^{\kappa'}(\ci):=\pi_{\kappa, \kappa'}\circ D_\rho^\mathcal{\mathcal{Z}}(\ci): \cE_{\kappa} (\ci)\rightarrow \cE_{\kappa\cdot\kappa'} (\ci).
\end{align*}

\subsubsection{$p$-adic differential operators on vector-valued automorphic forms}\label{Dwork-sec} 
We now consider the pullback of the universal abelian scheme $\cA^{\ord}/\Sord$ over $ \Ig$.
In analogue with the Hodge decomposition of $\hdr^1(\ci)$, there is a decomposition over $\Ig$
\begin{align*}
\hdr^1(\cA^{\ord}/\Ig) = \uo_{\cA^{\ord}/\Ig}\oplus\underline{U},
\end{align*}
where
$\underline{U}$ is Dwork's {\it unit root submodule}, introduced in \cite{kad}.
By \cite[Theorem (1.11.27)]{kaCM} and \cite[Proposition V.8]{EDiffOps}, $\nabla(\underline{U})\subset \underline{U}\otimes \Omega^1_{\igusa/\Witt}$.

As before, for each irreducible representation $\rho$ as above, we define
		 \[\varpi_{\rho}(\cA^{\ord}/\Ig): \hdr^1(\cA^{\ord}/\Ig)^{\rho} \twoheadrightarrow {\uo^{\rho}_{\cA^{\ord}/\Ig}}\]
to be the projection induced by Dwork's unit root decomposition 
after applying the Schur functor $(\,)^\rho$.
Note that $\uo_{\cA^{\ord}/\Ig}^{\rho}$ is identified  with the pullback of $\cE_{\rho}$ over $\Ig$ via the definition of Schur functors. 

Analogously to how we defined the $\ci$ differential operators $D_\rho^e\left(\ci\right)$,  $D_\rho^{\mathcal{Z}}\left(\ci\right)$, and $D_\kappa^{\kappa'}(\ci)$ in Section \ref{cidefns-section}, replacing $\overline{\uo(\ci)}$ by $\underline{U}$,
we define $p$-adic differential operators $D_\rho^e\left(\cA^{\ord}/\Ig\right)$,  
$D_{\rho}^{\cZ}(\cA^{\ord}/\Ig)$, and 
$D_\kappa^{\kappa'}(\cA^{\ord}/\Ig)$, for all $e$, $\cZ$, $\kappa$, $\kappa'$, and $\rho$ as above. 

In the sequel, for each sum-symmetric weight $\kappa'$ of depth $e$, and each positive dominant weight $\kappa$, we write
\[{D_\kappa^{\kappa'}}:=D_\kappa^{\kappa'}(\cA^{\ord}/\Ig)
: \cE_{\kappa} \rightarrow \cE_{\kappa\cdot\kappa'}.\]

\section{Localization at an ordinary point}\label{local-section}

The ultimate goal of this section is to describe the action of the differential operators on the $t$-expansions of $p$-adic automorphic forms.
We start by describing the constructions of Section \ref{diffop-sec} in terms of Serre--Tate local parameters, now taking $S=\Sord$ the ordinary locus of a connected component $\cS$ of the Shimura variety $\cM$, and $A/S$ the universal abelian scheme $A = \cA^{\mathrm{ord}}$ (as defined in the beginning of Section \ref{Igusalevel}).

Throughout the section,  we fix a point $x\in\Igusa(\Witt)$ lying above $\xzero\in \Sord(\Witt)$, and denote by $\bar{x},\bar{x}_0$  their reduction modulo $p$. In the following, we write $\cR$ for the complete local ring {$\Ring$ corresponding to $\Sord^\wedge_{\ov x_0}$ introduced in Section \ref{ST-coord}, 
 and $\fm_\cR=\fm_{\ov x_0}$ for its maximal ideal.  {We denote by $\cA_{\xzero}^\ord$ the universal formal deformation of the abelian variety {with additional structures} $A_0=A_{\xzeroreduced}$, i.e. $\cA_{\xzero}^\ord$ is the base change of $\cA^\ord$ from $\Sord$ to $\cR$ and has special fiber $A_0$. By abuse of notation, we also abbreviate $\cA_{\xzero}^\ord$ by $A$.}

\subsection{The Gauss--Manin connection}

In \cite{KaST} Katz explicitly describes the Gauss--Manin connection and Dwork's unit root submodule in terms of the Serre--Tate coordinates.  We recall his results. 

Let $\hat{\cH}$ denote the formal relative de Rham cohomology bundle, $\hat{\cH}=R^1\pi_*(\Omega^1_{A/\cR})$.
We write $\Phi$ for the $\cR$-semilinear action of Frobenius on $\hat{\cH}$, and 
\begin{equation}\label{dualexactsequence}
0\to \underline{\omega}_{A/\cR}\to\hat{\cH}\to \underline{\omega}^\vee_{A^\vee/\cR}\to 0
\end{equation}
for the (localized) Hodge exact sequence over $\cR$ (where we identified $H^1(A,\cO_A)$ with $ \underline{\omega}^\vee_{A^\vee/\cR}$ as in Equation \eqref{xdualsurj-map}).

\begin{prop} \cite[{Cor.~4.2.2}]{KaST}
\label{ST-Hodge}
Notation and assumptions are the same as above.  
\begin{enumerate}
\item There is a canonical Frobenius-equivariant isomorphism
\[\alpha: T_p(A_0^\vee)\otimes \cR\to  \underline{\omega}_{A/\cR},\] 
where the $\cR$-semilinear {action} of Frobenius on the left hand side is defined by extending multiplication by $p$ on $T_p(A_0^\vee)$.

\item There is a canonical Frobenius-equivariant isomorphism
\[\omicron: {\rm Hom}(T_pA_0,\zz_p)\otimes \cR \to \underline{\omega}^\vee_{A^\vee/\cR}\]
where the $\cR$-semilinear action of Frobenius on the left hand side is defined by extending the identity on ${\rm Hom}(T_pA_0,\zz_p)\otimes \cR$.

\item The surjection \[\hat{\cH}\to \underline{\omega}^\vee_{A^\vee/\cR}\simeq {\rm Hom}(T_pA_0,\zz_p)\otimes \cR\]  induces an isomorphism between the $\zz_p$-submodule $L_1$ of $\hat{\cH}$ where $\Phi$ acts trivially and ${\rm Hom}(T_pA_0,\zz_p)$.  The inverse of such an isomorphism defines a canonical splitting of the Hodge exact sequence over $\cR$,
$v: \underline{\omega}^\vee_{A^\vee/\cR}\to\hat{\cH},$
i.e. there is a canonical $\cR$-linear decomposition $\hat{\cH}= \underline{\omega}_{A/\cR}\oplus (L_1\otimes_{\zz_p}\cR).$
\end{enumerate}
\end{prop}

\begin{rmk}\label{u} The submodule $L_1\otimes_{\zz_p} \cR$ agrees with the {base change} $\underline{U}_{{x}_0}$
of Dwork's unit root submodule $\underline{U}$ to $\cR$ introduced in Section \ref{Dwork-sec}. In the following, we write $\underline{U}_{\cR}\subset \hat{\cH}$ for the the submodule $L_1\otimes_{\zz_p}\cR \subset \hat{\cH}$ and denote by $u:\hat{\cH}\to  \underline{\omega}_{A/\cR}$ the projection modulo $\underline{U}_{\cR}$.
\end{rmk}

\begin{rmk}
In our setting, the action of $\cO_K$ on the abelian scheme $A/\cR$ induces natural structures of $\cO_K$-modules on  $T_p(A_0^\vee),
 {\rm Hom}(T_pA_0,\zz_p)$, $\underline{\omega}_{A/\cR}$, and $\underline{\omega}^\vee_{A^\vee/\cR}$. It follows from the construction that the isomorphisms $\alpha$ and $\omicron$  are $\cO_K$-linear. 
\end{rmk}

By abuse of notation we will still denote by $\nabla$ the Gauss--Manin connection on $\hat{\cH}/\cR$,    \[\nabla:\hat{\cH}\to\hat{\cH}\otimes_\cR \Omega^1_{\cR/\Witt}.\]

In the following proposition,  we denote by $\fq$  the universal bilinear form on $\cR$ introduced in Equation \eqref{fq-equ}, and by $T_p(\lambda)$ the isomorphism of physical Tate modules, $T_pA_0\cong T_p(A^\vee_0)$, induced by the polarization $\lambda$. {Finally, for any $\zz_p$-basis $\fT$ of $T_pA_0$, we denote by $\{\delta_Q|Q\in\fT\}$ the associated dual basis of ${\rm Hom}(T_pA_0,\zz_p)$.}

\begin{prop} \label{horizontal} (\cite[{Thm.~4.3.1}]{KaST})
The notation is the same as in Proposition \ref{ST-Hodge}.
\begin{enumerate}

\item For each $\delta\in {\rm Hom}(T_pA_0,\zz_p)$, the differential $\eta_\delta:=v(\omicron(\delta))\in L_1$ satisfies $\nabla\eta_\delta=0$.\label{item1}

\item {For each $e\in T_p(A_0^\vee)$,  the differentials $\omega_e=\alpha(e)\in \underline{\omega}_{A/\cR}$ satisfy \[\nabla \omega_e =\sum_{Q \in \fT} \eta_{\delta_Q}\otimes d{\rm log }\,\fq(Q,T_p(\lambda)^{-1}(e)),\] for any $\zz_p$-basis $\fT$ of $T_pA_0$.}\label{item2}
\end{enumerate}
\end{prop}

Note that Part \eqref{item1} implies that $\nabla(\underline{U}_{\cR})\subset \underline{U}_{\cR}\otimes _\cR \Omega^1_{\cR/\Witt}$, 
as stated in Section \ref{Dwork-sec}. Also, Part \eqref{item2} implies that for each $e\in T_p(A_0^\vee)$,  the differentials $\omega_e$ satisfy 
$\nabla \omega_e\in \underline{U}_{\cR} \subset \hat{\cH},$ i.e., 
 $u(\nabla \omega_e)=0$.

\subsection{The Kodaira--Spencer morphism}  
In this section, we explicitly describe the Kodaira--Spencer morphism in terms of the Serre--Tate coordinates.  By abuse of notation we will still denote by ${\rm KS}$ the localization at the point $\xzero\in \Sord(\Witt)$
of the Kodaira--Spencer morphism, i.e.
\[{\rm KS}: \underline{\omega}_{A/\cR}\otimes_{\cR} \underline{\omega}_{A/\cR}\to \Omega^1_{\cR/\Witt}.\]

\begin{prop} \label{Prop-421}
 For all $P\in T_pA_0$, let $\omega_P:=\alpha(T_p(\lambda)(P))$. Then, for all $P,P'\in T_pA_0$ and any $\zz_p$-basis $\fT$ of $T_pA_0$,  
\[{\rm KS}(\omega_P {\otimes} \omega_{P'})=\sum_{Q\in \fT}(Q,P')_{\lambda}d{\rm log}\fq(Q,P) ,\] where $(,)_\lambda$ denotes the $\lambda$-Weil pairing on $T_pA_0$.\end{prop}

\begin{proof}
By definition, for any $\omega_1,\omega_2\in \underline{\omega}_{A/\cR}$,
\[{\rm KS}(\omega_1 {\otimes} \omega_2)= \langle (\pi\otimes\id)(\nabla\omega_1), \lambda(\omega_2)\rangle\in\Omega_{\cR/\Witt},\]
where  $\pi:\hat{\cH}\to \underline{\omega}_{{A}^\vee/\cR}^\vee$ is the projection in the Hodge exact sequence, $\lambda: \underline{\omega}_{A/\cR}\isomto  \underline{\omega}_{A^\vee/\cR}$ is the isomorphism induced by the polarization $\lambda$ on $A$, and  
$\langle,\rangle:  \left(\underline{\omega}_{A^\vee/\cR}^\vee{\otimes\Omega^1_{A/\cR}}\right)\times  {\underline{\omega}_{A^\vee/\cR}\to\Omega^1_{A/\cR}}$ is the map obtained {by} extending the natural pairing  $(,):\underline{\omega}_{A^\vee/\cR}^\vee\times  \underline{\omega}_{A^\vee/\cR}\to \cR$ by the identity map on  $ \Omega^1_{A/\cR}$.

Let us fix a basis $\fT$ of $T_pA_0$;
for all $P\in \fT$, we write $\eta_P:= \eta_{\delta_P} = v(a(\delta_P))$. 
 We deduce from the definitions and Proposition \ref{horizontal} that for all $P,Q\in \fT$, \[\langle \pi(\eta_Q),\lambda(\omega_P)\rangle=(Q,P)_\lambda, \text{ and }
\nabla\omega_P=\sum_{Q\in\fT}\eta_Q\otimes d{\rm log}\fq(Q,P).\]
Thus,  for all $P,P'\in \fT$, we have
\[{\rm KS}(\omega_P {\otimes} \omega_{P'})=\langle (\pi\otimes\id)\nabla(\omega_P),\lambda(\omega_{P'})\rangle=\]\[=\langle \sum_{Q\in\fT} \pi(\eta_Q),\lambda(\omega_{P'})\rangle  d{\rm log}\fq(Q,P)=\sum_{Q\in \fT}(Q,P')_{\lambda}d{\rm log}\fq(Q,P) .\]

\end{proof}

\begin{rmk}
Theorem \ref{ST-coord-thm} implies that for any $P\in T_{\fP_i}A_0$, $\fq(Q,P)=0$  unless $Q\in T_{\fP_i^c}A_0$. Thus, the morphism ${\rm KS}$ factors via the quotient $\underline{\omega}_{A/\cR}^2=\oplus_{\tau\in\arch}\left(\uo_{A/\cR, \tau}^+\otimes\uo_{A/\cR, \tau}^-\right)$, and Proposition \ref{ks} implies that the induced map is an isomorphism. 
\end{rmk}

\subsection{The differential operators}\label{dox}

Finally, in this section  we explicitly describe the differential operators  in terms of the Serre--Tate coordinates.  
By abuse of notation, we will still denote by $D_\rho$ (resp. $D_\rho^e$) the localization at $\xzero$ of the differential operators $D_\rho$ (resp. $D_\rho^e$) introduced in Section \ref{doperator}{, i.e. its base change to $\cR$}.

We briefly recall the constructions. 
Let  $u:\hat{\cH}\to \underline{\omega}_{A/\cR}$ denote the projection modulo $\underline{U}_\cR$ (as defined in  Remark \ref{u}).  We define 
\begin{align}\label{D-def}
D:=(u\otimes \id)\circ \left.\nabla\right|_{\underline{\omega}_{A/\cR}}: \underline{\omega}_{A/\cR}\to \underline{\omega}_{A/\cR}\otimes_{\cR} \Omega^1_{\cR/\Witt},
\end{align}
 where $\id$ denotes the identity map on $\Omega^1_{\cR/\Witt}$. 

By abuse of notation, we still denote by ${\rm ks}^{-1}:\Omega{^1}_{\cR/\Witt}\isomto 
\underline{\omega}_{A/\cR}^2\subset \underline{\omega}_{A/\cR}^{\otimes 2}$ the localization of the inverse of the Kodaira--Spencer isomorphism defined in Proposition \ref{etaleks-prop}. 

For all positive integers $d$ and $e$, we write 
$$D_{\otimes d}:= (u^{\otimes d}\otimes \id)\circ \left.\nabla_{\otimes d}\right|_{\underline{\omega}^{\otimes d}_{A/\cR}}:  \underline{\omega}^{\otimes d}_{A/\cR}\to \underline{\omega}^{\otimes d}_{A/\cR}\otimes_{\cR} \Omega^1_{\cR/\Witt},$$
$D^1_{\otimes d}:= (\id_{\underline{\omega}^{\otimes d}_{A/\cR}} \otimes {\rm ks}^{-1})\circ  D_{\otimes d}
$ and $D^e_{\otimes d}:=D^1_{\otimes d+2e-2}\circ \cdots \circ  D^1_{\otimes d}$.

Let  $\cE_{\rho,\xzero}$ denote the localization of $\cE_{\rho}$, i.e. 
 the base change to $\cR$.  For any {irreducible} representation $\rho=\rho_\kappa$, and positive integer $e$, 
the differential operator
\[D^e_\rho:\cE_{\rho,\xzero} \to\cE_{\rho, \xzero} \otimes_\cR (\underline{\omega}_{A/\cR}^2)^{\otimes e}\]
is {induced by the restrictions} of $D^e_{\otimes d_\kappa}$ {to $\uo^\rho_{A/\cR}$}. 
In the following, we also write  $D_\rho=D^1_\rho$.

\subsubsection{Local description}
We fix a point $x\in\Igusa(\Witt)$ lying above $\xzero$, and define $\alpha_x$ to be the $\cO_K\otimes_{\zz} \cR$-linear isomorphism
 \[\alpha_x:= \alpha\circ (T_p(\iota_x^\vee)^{{\vee}}\otimes\id):  \cL\otimes_{\zz_p} \cR \isomto T_pA_0^\vee\otimes_{\zz_p} \cR\to\underline{\omega}_{A/\cR},\]
where $T_p(\iota_x^\vee): T_pA_0\to \cL^\vee$ is defined as in \eqref{TateLattice} and  $\alpha:   T_p{(}A_0^\vee{)}\otimes_{\zz_p} \cR\to\underline{\omega}_{A/\cR} $  {as} in Proposition \ref{ST-Hodge}(1), and we
identify $T_p(A_0^\vee)$ with $(T_pA_0)^\vee$ via the Weil pairing.

By linearity, we deduce that the isomorphism $\alpha_x$ induces isomorphisms
\[\alpha_{x,\tau}^+: \cL^+_\tau\otimes_{\zz_p}\cR\isomto \underline{\omega}_{A/\cR,\tau}^+ \text{ and }
\alpha_{x,\tau}^-: \cL^-_\tau\otimes_{\zz_p}\cR\isomto \underline{\omega}_{A/\cR,\tau}^-,\]
for each $\tau \in\Sigma$.
We write 
\[\alpha_x^2:=\bigoplus_{\tau \in\Sigma} (\alpha_{x,\tau}^+\otimes \alpha_{x,\tau}^-):\cL^2\otimes_{\zz_p}\cR \isomto \underline{\omega}_{A/\cR}^2,\]
(recall  $\cL^2=\bigoplus_{\tau\in \Sigma}(\cL^+_\tau\otimes_{\zz_p}\cL^-_\tau)$),
and
\[{\rm ks}_x:={\rm ks}\circ \alpha_x^2: \cL^2\otimes_{\zz_p} \cR\isomto \Omega^1_{\cR/\Witt}.\] 
For any {irreducible} representation $\rho$, the map $\alpha_x$ also induces an $\cO_K\otimes_{\zz} \cR$-linear isomorphism
\[\alpha_x^\rho:\cL^\rho\otimes_{\zz_p}\cR \isomto \cE_{\rho,\xzero} {,}\]
{via the identification of $\cE_{\rho,\xzero}$ with $\uo^\rho_{A/\cR}$ defined by $x$.}

  Finally, for all $e\in\bN$, we define  $\alpha_x^{\rho,e}:=\alpha_x^\rho \otimes_\cR (\alpha^2_x)^{\otimes e}$,
\[ \alpha_x^{\rho,e}:
\cL^\rho\otimes_{\zz_p}(\cL^2)^{\otimes e}\otimes_{\zz_p} \cR =(\cL^\rho\otimes_{\zz_p}\cR)\otimes_{\cR} ((\cL^2)^{\otimes e}\otimes_{\zz_p}\cR) \isomto \cE_{\rho,\xzero}\otimes_\cR (\underline{\omega}_{A/\cR}^2)^{\otimes e}.\]

Let $d:\cR\to\Omega{^1}_{\cR/\Witt}$ denote the universal $\Witt$-derivation on $\cR$. We define
\[\Xi:={\rm ks}^{-1}_x\circ d:\cR\to \cL^2\otimes\cR.\]
For any integer $e\in\bN$, we write $$\Xi^e:={(\id_{(\cL^2)^{\otimes(e-1)}} \otimes \Xi)\circ \hdots \circ \Xi:\cR \to (\cL^2)^{\otimes e}\otimes_{\zz_p} \cR}.$$

\begin{prop}\label{Disd} For any {irreducible} representation $\rho$, and any integer $e\in\bN$,

\[({ \alpha^{\rho, e}_x})^{-1}\circ D^e_\rho\circ \alpha^\rho_x = 
\id\otimes \Xi^{e}:\cL^\rho\otimes_{\zz_p} \cR\to \cL^\rho\otimes_{\zz_p}  (\cL^2)^{\otimes e}\otimes_{\zz_p} \cR,\]

\end{prop}
\begin{proof}
Proposition \ref{horizontal} implies $(\alpha\otimes \id)^{-1}\circ D\circ \alpha=\id\otimes d$, with $D$ defined as in \eqref{D-def}. We deduce that 
 \[( \alpha_x\otimes \id)^{-1}\circ D\circ \alpha_x = \id\otimes d :\cL\otimes_{\zz_p} \cR\to \cL\otimes_{\zz_p} \Omega^{1}_{R/\Witt}.\]
Therefore, for any  representation $\rho$, we have
\[(\alpha_x^\rho\otimes \id)^{-1}\circ D_\rho \circ \alpha^\rho_x=\id\otimes {\rm ks}^{-1}\circ d:\cL^\rho\otimes_{\zz_p}\cR\to \cL^\rho\otimes_{\zz_p}  \underline{\omega}^2_{A/\cR} ,\]
and thus also ${(\alpha^{\rho, 1}_x)^{-1}}\circ D_\rho\circ \alpha^\rho_x = 
\id\otimes \Xi$. 

The general case, for $e\geq 2$, follows from the case $e=1$.
\end{proof}

\subsubsection{Explicit description of $\Xi$ in terms of Serre--Tate coordinates}\label{section-basisdef}
We conclude this section with an explicit description of the map $\Xi$ in terms of Serre--Tate coordinates.
Let $\fB$ be a $\zz_p$-basis of $\cL$ such that $\fB=\cup_{\tau\in \Sigma}(\fB_\tau^+\cup \fB_\tau^-)$ where, for all $\tau\in\Sigma$, 
{$\fB_\tau^{+} = \{\fb_{\tau,1},\dots,\fb_{\tau,a_{+ \tau}}\}$ is a $\zz_p$-basis of $\cL^+_\tau$ and $\fB_{\tau}^- = \{\fb_{\tau,a_{+\tau}+1}, \dots, \fb_{\tau,n}\}$ is a $\zz_p$-basis of $\cL^-_\tau$ such that} $\fB_\tau^\pm$ and $ \fB_{\tau^c}^\mp$ are dual to each other  under the Hermitian pairing on $\cL$.
We denote by $\fE = \cup_{\tau \in \Sigma} \fE_{\tau}$ (resp. $\fE^\vee$, $\fE_{e}=\fE\times \cdots \times \fE$ $e$ times, and $\fE^{\vee}_e=\fE^{\vee}\times \cdots \times \fE^{\vee}$ $e$ times)  the associated bases of ${\cL^2 = \bigoplus_{\tau \in \Sigma} \cL^+_{\tau} \otimes \cL^-_{\tau}}$ (resp. $(\cL^2)^\vee$,  $(\cL^2)^{\otimes e}$, and $((\cL^2)^{\vee})^{\otimes e}=((\cL^2)^{\otimes e})^\vee$).  Explicitly, 
	$$\fE_{\tau} = \{l^{\tau}_{i,j} := \fb_{\tau,i} \otimes \fb_{\tau,j} : 0 < i \leq a_{+\tau} < j \leq n\}.$$
Note that the pairing on $\cL$ induces a canonical isomorphism 
$\cL^2\isomto (\cL^2)^\vee$, which identifies $\fE$ with $\fE^\vee$. In the following, by abuse of notation we write  $l\mapsto l^\vee$, for both the map $\fE\to\fE^\vee$ and its inverse.

Let $\beta^*_{x,\fE^\vee}:\Witt[[t_l|l\in\fE^\vee]] \isomto \cR $ denote the Serre--Tate isomorphism associated with the choice of $x$ and $\fE^\vee$, as defined in Remark  \ref{beta-rmk}.
Recall  that, for all $l\in\fE^\vee$, we have $\beta^*_{x,\fE^\vee}(t_l)= \fq(t_x^{-1}(l))-1$,  
with $t_x=T_p(\iota_{{x}}{^\vee})^{\otimes 2}$.

\begin{prop} \label{Delta}
The notation is the same as above. 
For all $f\in\cR$  and $k\in \fE^\vee$, we have 
\[\beta_{x,\fE^\vee}^{*-1}((k\otimes \id)(\Xi (f)))=(1+t_k)\partial_{k} \beta_{x,\fE^\vee}^{*-1}(f)\in \Witt[[t_l|l\in\fE^\vee]],\]
where $\partial_{k}:=\frac{\partial}{\partial t_k}$ denotes the partial derivation with respect to the variable $t_k$.
\end{prop}

\begin{proof}
From the definition of $\Xi$ follows that it suffices to prove the equalities
\[\beta_{x,\fE^\vee}^{*-1}(k\otimes \id)(\Xi(\beta^*_{x,\fE^\vee}(t_l)))=(1+t_l)\delta_{l,k}\] where $\delta_{l,k}$ denotes the Kronecker symbol (i.e. $\delta_{lk}=1$ if $l=k$, and $0$ otherwise), for all $k,l\in \fE^\vee$.  {We have (using Proposition \ref{Prop-421})}
 \[\Xi(\beta^*_{x,\fE^\vee}(t_l))=(\alpha_x^2)^{-1}{\rm ks}^{-1} d( \fq(t_x^{-1}(l))-1)=(\alpha_x^2)^{-1}{\rm ks}^{-1}( d\fq(t_x^{-1}(l)))=\]
\[=(\alpha_x^2)^{-1}{\rm ks}^{-1} (\fq(t_x^{-1}(l)) d {\rm log} \fq(t_x^{-1}(l)))=
l^\vee\otimes \fq(t_x^{-1}(l)),  
\]
which implies
\begin{eqnarray*}
\beta_{x,\fE^\vee}^{*-1}(k\otimes \id)(\Xi(\beta^*_{x,\fE^\vee}(t_l))) &=& \beta_{x,\fE^\vee}^{*-1}(k\otimes \id)(l^\vee\otimes \fq(t_x^{-1}(l))) \\ &=&   \beta_{x,\fE^\vee}^{*-1}(\delta_{l,k}\fq(t_x^{-1}(l)))\\ &=& (1+t_l)\delta_{l,k}.
\end{eqnarray*}
\end{proof}

\section{Main results on $p$-adic differential operators}\label{mainresultsops-section}
In this section, we construct $p$-adic differential operators on Hida's space $V^N$ of $p$-adic automorphic forms by interpolating the differential operators defined in Section \ref{diffop-sec}.  First, using the $t$-expansion principle and Theorem \ref{density}, we prove that the $p$-adic differential operators on  the space of classical automorphic form extend uniquely to all of $V^N$ (Theorem \ref{Theta}). Secondly, we prove that $p$-adic differential operators of congruent weights are congruent (Theorem \ref{congruence}). As a corollary, we establish the existence of $p$-adic differential operators of $p$-adic weights interpolating those of classical weights (Corollary \ref{cong-coro}).

\subsection{$p$-adic differential operators of classical weights}
\label{coord-section}
In this section we prove that the differential operators 
\[D_{\kappa'}^{\kappa}:\cE_{\rho_{\kappa'}}\to\cE_{\rho_{\kappa'\cdot \kappa}},\] (where $\kappa'$ is a positive dominant weight and $\kappa$ is a sum-symmetric weight) induce differential operators $\Theta^{\chi}$ on the space of $p$-adic automorphic forms $V^N$, satisfying the property $\Theta^{\chi}\left(V^N[\chi']\right)\subset V^N[\chi'\cdot \chi]$ for all $p$-adic weights $\chi'$.

In the following, we write $\cE_\kappa$ (resp. $\cL^\kappa$, $\alpha_x^\kappa$, \ldots) in place of $\cE_{\rho_\kappa}$ (resp. $\cL^{\rho_\kappa}$, $\alpha^{\rho_\kappa}_x$, $\ldots$).  By abuse of notation, we still write $D^{\kappa}_{\kappa'}$ in place of the map on global sections \[D^{\kappa}_{\kappa'}(\Sord): H^0(\Sord,\cE_{\kappa'})\to H^0(\Sord, \cE_{\kappa'\cdot\kappa}).\]

For any weight $\kappa'$, we write $\Psi_{\kappa'}:H^0(\Sord,\cE_{\kappa'})\hookrightarrow V^N[\kappa']\subset V$ as in \eqref{psik-map}. By definition, the localization  of $\Psi_{\kappa'}$ at the point $x\in\Igusa(\Witt)$ agrees with the map {\[(\lcan^{\kappa'}\otimes {\rm id})\circ (\alpha_x^{\kappa'})^{-1}:\cE_{\kappa',\xzero}\to\cL^{\kappa'}\otimes_{\zz_p}\cR\to \cR .\]}
 where { $\lcan^{\kappa'}: \cL^{\kappa'}\to\zz_p$ is defined as in  Definition \ref{lcan-def}}, and {${\rm id}$ denotes the identity of $\cR$}.

By abuse of notation, for any sum-symmetric weight $\kappa$ of depth $e$, we still denote by
\[\tilde{\ell}_{\rm can}^{\kappa}: (\cL^2)^{\otimes e} \to \zz_p\] the map induced by $\tilde{\ell}_{\rm can}^{\kappa}:\cL^{\otimes 2e}\to  \zz_p$ as defined in Definition \ref{lcan-def} (recall that $(\cL^2)^{\otimes e}$ is a direct summand of $\cL^{\otimes 2e}$,  see also Remark \ref{sum-sym}). We write
$\tilde{\ell}_{\rm can}^{\kappa}\otimes {\rm id}:(\cL^2)^{\otimes e}\otimes_{\zz_p}\cR\to \cR$ for the associated $\cR$-linear map.

\begin{defn} \label{theta-def} 
For any positive integer $e\in\bN$, and any sum-symmetric weight $\kappa$ of depth $e$, we define  \[\theta^{\kappa}:=(\tilde{\ell}_{\rm can}^{\kappa}\otimes {\rm id}) \circ \Xi^e  :\cR\to(\cL^2)^{\otimes e}\otimes_{\zz_p} \cR\to \cR.\]
We call $\theta^{\kappa}$ the $\kappa$-differential operator on Serre--Tate expansions.
\end{defn}

We fix a point  $x\in \Igusa(\Witt)$, and write  ${\rm loc}_x:V^N\to \cR$ for the localization map at $x$ { as introduced in Section \ref{section-t-expansion}.}

\begin{lem}\label{formula-lemma}
For any weights $\kappa,\kappa'$, where $\kappa'$ is  positive dominant and $\kappa$ is sum-symmetric, and for all $f\in H^0(\Sord,\cE_{\kappa'})$, we have
\[\theta^{\kappa}({\rm loc}_x( \Psi_{\kappa'}(f)))=
{\rm loc}_x(\Psi_{\kappa'\cdot \kappa}(D_{\kappa'}^{\kappa}(f)) . \]
\end{lem}
\begin{proof}
{Let  $\pi_{\kappa',\kappa}: \cL^{\kappa'} \otimes \cL^{\kappa} \ra \cL^{\kappa' \cdot\kappa}$ be defined as in Lemma \ref{Lemma-lcan-composition}. By Equation \eqref{eqn-lcan-composition} we have $\lcan^{\kappa'}\otimes \lcan^{\kappa}= \lcan^{\kappa'\cdot\kappa}\circ \pi_{\kappa',\kappa}:\cL^{\kappa'}\otimes\cL^{\kappa}\to \cL^{\kappa'\cdot \kappa}\to \cR$, and the lemma then follows }from Proposition \ref{Disd}.
\end{proof}

\begin{thm}\label{indept-thm}\label{Theta}
For each sum-symmetric weight $\kappa$,
there exists a unique operator
\begin{align*}
\Theta^{\kappa}: V^N\rightarrow V^N
\end{align*}
such that $\Theta^\kappa \circ \Psi := \Psi \circ D_{\kappa'}^{ \kappa}.$

The $p$-adic $\kappa$-differential operator $\Theta^\kappa$ satisfies the  properties:
\begin{enumerate}
\item for all $f\in V^N$: ${\rm loc}_x\circ \Theta^\kappa=\theta^\kappa\circ {\rm loc}_x$,
\item for all weights $\kappa'$: $\Theta^{\kappa}(V^N[\kappa'])\subset V^N[\kappa'\cdot \kappa].$
\end{enumerate}

\end{thm}
\begin{proof}
Using the fact that $\Psi$ is an injection, we first define $\Theta^\kappa$ on the image of $\Psi$ in $V$ by
\begin{align*}
\Theta^\kappa(f) := \Psi \circ D_{\kappa'}^{\kappa}\circ \Psi^{-1}(f)
\end{align*}
for each $f\in {\im(\Psi_{\kappa'})}${ for all positive dominant weights $\kappa'$.}
Since  $$ \Psi\left( \bigoplus_{\kappa \in X(T)_+, \atop \kappa \text{ positive }} H^0(\Sord,\Ekap)\right)\left[\frac{1}{p}\right] \cap V^N$$  is dense in $V^N$,
it is clear that if $\Theta^\kappa$ exists, then it is unique. In order to prove that indeed $\Theta^\kappa$ extends to {\it all} {of} $V^N$, it is sufficient to
check that  if $f_1, f_2, \ldots\in \Im(\Psi)$ converge to an element $f\in V^N$, then $\Theta^\kappa(f_1), \Theta^\kappa(f_2), \ldots$ converge in $V^N$ to $\Theta^\kappa(f) \in V^N$.

By the Serre--Tate expansion principle (Theorem \ref{t-exp}), one can check convergence after passing to $t$-expansions, in which case the statement  follows from Lemma  \ref{formula-lemma}. 
Properties (1) and (2) follow immediately from the construction.
\end{proof}

\begin{rmk}
The operators $\Theta^\kappa$ play a role analogous to the role played by Ramanujan's theta operator in the theory of modular forms and Katz's theta operator in the theory of Hilbert modular forms (see \cite[Remark (2.6.28)]{kaCM}). 
\end{rmk}

\subsection{$p$-adic differential operators of $p$-adic weights}\label{p-weights-section}
In this section we establish congruence relations for the differential operators $\Theta^\kappa$ as $\kappa$ varies. As an application we deduce {the existence of} $p$-adic differential operators $\Theta^\chi$ for $p$-adic characters $\chi$, interpolating operators of classical weights.

In the following, we fix a $\zz_p$-basis $\fB^\vee =\cup_\tau (\{\fb^\vee_{\tau,1},\dots, \fb^\vee_{\tau,a_{+\tau}}\} \cup \{\fb^\vee_{\tau,a_{+\tau}+1},\dots, \fb^\vee_{\tau,n}\})$ of $\cL^\vee$ as in Section \ref{section-basisdef},
and write $\fE^\vee$  (resp. $\fE^\vee_e$) for the associated basis of $(\cL^2)^\vee$ (resp. $((\cL^2)^{\otimes e})^\vee$).

\begin{rmk}\label{aipm}
For any sum-symmetric weight $\kappa=(\kappa_1^\tau, \hdots, \kappa_n^\tau)$, Equation \eqref{eqn-lcan-formula}, i.e. 
	\begin{equation} \label{eqn-tildelcan1}
		\tilde{\ell}_{\rm can}^{\kappa}=\prod_{\tau \in \Sigma} \prod_{i=1}^n (\kappa_i^\tau !)^{-1} \cdot \bigotimes_{\tau \in \Sigma} \bigotimes_{i=1}^{n} (\fb_{\tau,i}^\vee)^{\otimes \kappa_i^\tau} \cdot c_\kappa,
	\end{equation}
implies that $ \tilde{\ell}_{\rm can}^{\kappa}$ is a linear combination of elements of the basis $\fE_e^\vee$ of $((\cL^2)^\vee)^{\otimes e}$ with coefficients in $\{\pm 1\}$, for $e=\sum\limits_{\tau \in \Sigma} \sum\limits_{i=1}^{\siga_\tau}\kappa_i^\tau$ the depth of $\kappa$.

For all $\underline{l}=(l_1,\dots ,l_e)\in\fE^\vee_e$, we define $a_{\kappa,\underline{l}}\in \{0, \pm 1\}$ such that
$ \tilde{\ell}_{\rm can}^{\kappa}=\sum_{\underline{l}\in\fE^\vee_e} a_{\kappa,\underline{l}}\cdot \underline{l}.$
\end{rmk}

We choose a point $x\in \Igusa(\Witt)$, and write
$ \beta_{x,\fE^\vee}^{*}:\Witt[[t_l|l\in\fE^\vee]]\stackrel{\sim}{\longrightarrow}\cR$ for the Serre--Tate isomorphism at the point $x$, written with respect to the $\zz_p$-basis $\fE^\vee$ of $(\cL^2)^\vee$.

\begin{lem}\label{theta}The notation is the same as above.
For any sum-symmetric weight $\kappa$ of depth $e${, and for all $f\in \cR$, we have $\beta_{x,\fE^\vee}^{*-1}(\theta^{\kappa}(f))$ is equal to}
\[\sum_{\underline{l}=(l_1,\dots ,l_e)\in\fE^\vee_e} a_{\kappa,\underline{l}} (1+t_{l_e})\partial_{l_e}\left(\dots (1+t_{l_2})\partial_{l_2}\left((1+t_{l_1})\partial_{l_1}( \beta_{x,\fE^\vee}^{*-1}(f))\right)\dots\right) .
\]
\end{lem}

\begin{proof}
The equality follows from the definitions and Proposition \ref{Delta}. 
\end{proof}

\subsubsection{Congruences and actions of $p$-adic differential operators}\label{congruences-actions-section}
We now prove a 
lemma describing properties of certain differential operators $\theta^{\underline{d}}$ that are closely related to those appearing in the right-hand side of the equality in Lemma \ref{theta}.  For each $\underline{l} = (l_1,\dots,l_e) \in \fE_e^\vee$, we define
\begin{eqnarray*}
\theta_e^{\underline{l}}: \Witt[[t_l|l\in\fE^\vee]] &\rightarrow& \Witt[[t_l|l\in\fE^\vee]] \\
\beta &\mapsto& (1+t_{l_e})\partial_{l_e}\left(\dots (1+t_{l_2})\partial_{l_2}\left((1+t_{l_1})\partial_{l_1}( \beta )\right)\dots\right).
\end{eqnarray*}
For convenience, we first introduce some more notation.  For all $l^{\tau}_{i,j} \in \fE^\vee$ {as in Section \ref{section-basisdef}}, we define
	\begin{equation} \label{eq-diffopdef} \theta^{\tau}_{i,j}: \Witt[[t_l|l\in\fE^\vee]] \rightarrow \Witt[[t_l|l\in\fE^\vee]] \quad f \mapsto (1+t_{l^{\tau}_{i,j}})\frac{\partial}{\partial t_{l^{\tau}_{i,j}}} f.
	\end{equation}
Note that these operators commute, i.e. for $l^{\tau}_{i,j}$ and $l^{\tau'}_{i',j'} \in \fE^\vee$, $\theta^{\tau}_{i,j} \circ \theta^{\tau'}_{i',j'} = \theta^{\tau'}_{i',j'} \circ \theta^{\tau}_{i,j}$.

For all $\underline{d}{=(d^{\tau}_{i,j})_{l^{\tau}_{i,j}\in \fE^\vee}} \in \bZ_{\geq0}^{|\fE^\vee|}$, we define $\theta^{\underline{d}}$ as the composition of the $d^{\tau}_{i,j}$-th iterates of $\theta^{\tau}_{i,j}$ for all $l^{\tau}_{i,j} \in \fE^\vee$ (we take $(\theta^{\tau}_{i,j})^0 = \id$). By commutativity, the order does not matter.}
For each $\ul l \in \fE_e^\vee$, we define $\ul d (\ul l)$ to be the tuple of non-negative integers such that 
	\begin{equation} \label{eqn-d-of-l}
		\theta^{\ul d(\ul l)}=\theta^{\ul l}_e.
	\end{equation}

\begin{rmk}\label{poly-rmk}

Let $R$ denote the subring of polynomials in $\cR$, i.e. $R:=\Witt[t_l|l\in\fE^\vee]\subset \cR=\Witt[[t_l|l\in\fE^\vee]]$. The differential operators $\theta^{\underline d}$ on $\cR$  are continuous for the  $(t_l|l\in\fE^\vee)$-adic topology on $\cR$, and preserve the subring $R$, i.e. $\theta^{\underline d} (R)\subseteq R$. 
In particular,  congruences among the operators
$\theta^{\underline d}$ on $\cR$ can be detected by studying congruences among their restrictions to $R$.

Furthermore, for all polynomials $f(\underline{t})\in R$, if we write $f(\underline{t}) = \sum_\alpha c_\alpha (1+\underline{t})^\alpha$, 
	where $(1+\underline{t})^\alpha := \prod_\tau\prod_{i, j}(1+t_{l_{i, j}^\tau})^{\alpha_{j, i}^\tau}$ for a collection of numbers $\alpha_{j, i}^\tau\in\bZ_{\geq 0}$, then 
	\begin{equation}\label{poly}
	(\theta^{\underline{d}} f)(\underline{t}) = \sum_\alpha \phi_{\underline{d}}(\alpha) c_\alpha (1+\underline{t})^\alpha,
	\end{equation} where $\phi_{\underline{d}}$ is a polynomial (in the numbers $\alpha_{j, i}^\tau$) dependent on $\underline{d}$.  (We set $\alpha:=\left(\alpha^\tau\right)_\tau$, with $\alpha^\tau = \left(\alpha_{j, i}^\tau\right)$.) 
Therefore,  congruences among the operators
$\theta^{\underline d}$ on $R$ can be detected by studying congruences among the polynomials $\phi_{\underline d}$.
(Note that Formula \eqref{poly} does not extend to $\cR$ as in general an element of $\cR$ cannot be written as a power series in $(1+\underline{t})$.)

\end{rmk}


\begin{prop}\label{cong1}
Let $\kappa,\kappa'$ be two {symmetric} weights (as in Definition \ref{sym-def}) and let $m \geq 1$ be an integer. Assume 
$$\kappa \equiv \kappa' \mod p^m(p-1)$$ in $\bZ^{rn}.$
Additionally, if
\begin{enumerate}[(i)]
\item $\min(\kappa_{i}^\tau-\kappa_{i+1}^\tau,{\kappa'}_i^\tau-{\kappa'}_{i+1}^\tau) > m  
  \text{ for all } \tau \in \Sigma \text{ and }  1 \leq i < \siga_\tau\text{ for which }  
  {\kappa_{i}^\tau-\kappa_{i+1}^\tau\neq{\kappa'}_i^\tau-{\kappa'}_{i+1}^\tau,}
  \text{ and}$
\item $\min(\kappa_{\siga_\tau}^\tau,{\kappa'}_{\siga\tau}^\tau) > m  \text{ for all } \tau \in \Sigma \text{ for which } \kappa_{\siga_\tau}^\tau \neq {\kappa'}_{\siga_\tau}^\tau,$
\end{enumerate}
then $\theta^{\kappa} \equiv \theta^{\kappa'} \mod p^{m+1}.$
\end{prop} 

\begin{proof}
	 By Lemma \ref{theta} and Equation \eqref{eqn-d-of-l}, we obtain 
	\begin{equation} \label{eqn-tildelcan} \beta_{x,\fE^\vee}^{*-1}\circ \theta^{\kappa} \circ \beta_{x,\fE^\vee}^{*}= \sum_{\underline{l}\in\fE^\vee_e} a_{\kappa,\underline{l}}\cdot \theta^{\ul d(\underline{l})} \, .  \end{equation}
	
	We assume without loss of generality that 
	$$\kappa'=\kappa \cdot \left(\eps^\tau_1 \eps^\tau_{\siga_\tau+1} \eps^\tau_2 \eps^\tau_{\siga_\tau+2}  ... \eps^\tau_{i}\eps^\tau_{\siga_\tau+i}\right)^{p^m(p-1)}$$
	 for some $\tau \in \Sigma$. (Recall that $\varepsilon^\tau_j$ denotes the character of $T(\zz_p)$ given by $\eps^{\tau}_j(\diag(\gamma^{\sigma}_{1,1}, \hdots, \gamma^{\sigma}_{n,n})_{\sigma \in \arch}) = \gamma^\tau_{j,j}$,
	i.e. such that $\kappa^{\tau}=\prod_j(\eps^{\tau}_j)^{\kappa^{\tau}_j}.$) 
	By combining Equations \eqref{eqn-tildelcan} and \eqref{eqn-tildelcan1}, and analyzing the action of the generalized Young symmetrizer, 
	 we obtain that
	 
\begin{align}
\beta_{x,\fE^\vee}^{*-1}\circ \theta^{\kappa'} \circ \beta_{x,\fE^\vee}^{*} &= \sum_{\underline{l}\in\fE^\vee_e}  a_{\kappa,\underline{l}} \theta^{\ul d(\underline{l})}\left(i!\cdot \left(\sum_{\sigma \in \Symm_i} (-1)^{\sgn(\sigma)}\cdot\prod_{j=1}^i \theta^\tau_{j,\siga_\tau+\sigma(j)}\right)\right)^{p^m(p-1)} \label{eqn-determinant} \\
&\equiv \sum_{\underline{l}\in\fE^\vee_e}  a_{\kappa,\underline{l}} \theta^{\ul d(\underline{l})}\label{eqn-determinant-secondline} \\
&=\beta_{x,\fE^\vee}^{*-1}\circ \theta^{\kappa} \circ \beta_{x,\fE^\vee}^{*}  \mod p^{m+1},\nonumber 
\end{align}
where congruence \eqref{eqn-determinant-secondline} follows from the following observation.
   
In the notation of Remark \ref{poly-rmk}, 
\begin{align*}
\left(i!\cdot \left(\sum_{\sigma \in \Symm_i} (-1)^{\sgn(\sigma)}\cdot\prod_{j=1}^i \theta^\tau_{j,\siga_\tau+\sigma(j)}\right)\right)^{p^m(p-1)}
\end{align*}
 has the effect of multiplying each polynomial $\phi_{\underline{d}(\underline{l})}(\alpha)$ by 
 \[
   \left(i!\cdot \left(\sum_{\sigma \in \Symm_i} (-1)^{\sgn(\sigma)}\cdot\prod_{j=1}^i \alpha^\tau_{\siga_\tau+\sigma(j), j}\right)\right),^{p^m(p-1)}\]
which is congruent to 
\begin{align*}\equiv
      \left\{ \begin{array}{ll}    1\mod p^{m+1} & \text{ if each $\alpha^\tau_{\siga_\tau+\sigma(j), j}$ is relatively prime to $p$} \\      	
      	0\mod p^m & \text{ otherwise} ,
     \end{array} \right.
 \end{align*}
because  $p > n$ (though note that $p > \max_{\tau \in \Sigma}\{\min(\siga_\tau,\sigb_\tau)\}$ is enough).  
 If some $\alpha^\tau_{\siga_\tau+\sigma(j), j}$ is divisible by $p$, then $\phi_{\underline{d}(\underline{l})}(\alpha)\equiv 0\mod p^{m+1}$ by assumptions (i) and (ii).
 
Hence, we conclude that $\theta^{\kappa} \equiv \theta^{\kappa'} \mod p^{m+1}$ for all {symmetric} weights $\kappa$ and $\kappa'$ satisfying the above hypotheses.
\end{proof}

\begin{rmk}
	Note that if $\kappa$ is sum-symmetric, but not symmetric, then $\theta^{\kappa}=0$. This follows by combining Equation \eqref{eqn-tildelcan} and \ref{eqn-tildelcan1} together with an analysis of the action of the generalized Young symmetrizer $c_\kappa$.
\end{rmk}

\begin{thm}\label{congruence}
Let $\kappa,\kappa'$ be two symmetric weights and $m \geq 1$ be an integer. Assume 
$$\kappa \equiv \kappa' \mod p^m(p-1)$$ in $\bZ^{rn}.$
Additionally, if
\begin{itemize}
\item $\min(\kappa_{i}^\tau-\kappa_{i+1}^\tau,{\kappa'}_i^\tau-{\kappa'}_{i+1}^\tau) > m  
  \text{ for all } \tau \in \Sigma \text{ and }  1 \leq i < \siga_\tau\text{ for which }  {\kappa_{i}^\tau-\kappa_{i+1}^\tau\neq{\kappa'}_i^\tau-{\kappa'}_{i+1}^\tau,}
  \text{ and}$
  \item $\min(\kappa_{\siga_\tau}^\tau,{\kappa'}_{\siga\tau}^\tau) > m  \text{ for all } \tau \in \Sigma \text{ for which } \kappa_{\siga_\tau}^\tau \neq {\kappa'}_{\siga_\tau}^\tau,$
\end{itemize}
then $\Theta^{\kappa} \equiv \Theta^{\kappa'} \mod p^{m+1}.$
\end{thm}
\begin{proof}
By the $p$-adic Serre--Tate expansion principle (Theorem \ref{t-exp}), combined with Property (1) in Theorem \ref{Theta}, $\Theta^\kappa\equiv \Theta^{\kappa'} \mod p^{m+1}$ if and only if $\theta^\kappa\equiv \theta^{\kappa'} \mod p^{m+1}$. 
Then the statement follows from Proposition \ref{cong1}.
\end{proof}

\begin{defi} We define a \textit{(symmetric) $p$-adic character} to be a continuous group homomorphism $T(\bZ_p) \ra \bZ_p^*$ that arises as the $p$-adic limit of the $\bZ_p$-points of characters corresponding to (symmetric) weights.
\end{defi}

	Proposition \ref{cong1} enables us to define by interpolation the differential operators $\theta^\chi$  on $\cR$ for all symmetric $p$-adic characters $\chi$. Note that in order to define $\theta^\chi$,  one should only take
the limit over $\theta^\kappa$  for weights $\kappa$ that satisfy
conditions (i) and (ii) of Proposition \ref{cong1} (without loss of generality, one may also choose the weights $\kappa$ such that $|\kappa|_\infty\rightarrow \infty$). These conditions can always be
achieved by modifying the characters converging to $\chi$ if necessary.

The following result is then an immediate consequence of Theorem \ref{congruence}.

\begin{cor}\label{cong-coro} 
For all {symmetric} $p$-adic characters $\chi$, 
there exist  $p$-adic differential operators
\begin{align*}
\Theta^{\chi}: V^N\rightarrow V^N
\end{align*}
interpolating the $p$-adic $\kappa_i$-differential operators previously defined for classical weights $\kappa_i$, and satisfying the following proerties:
\begin{enumerate}
\item for all $f\in V^N$: ${\rm loc}_x\circ \Theta^\chi(f)=\theta^\chi\circ {\rm loc}_x(f)$,
\item {for all $p$-adic characters $\chi'$: $\Theta^{\chi}(V^N[\chi'])\subset V^N[\chi'\cdot \chi].$}
\end{enumerate}
\end{cor}

\subsubsection{Polynomials $\phi_\kappa$}  The remainder of this section introduces some notation and results needed in Section \ref{appli}.

\begin{defi}\label{phikap-action}
	For each  sum-symmetric weight $\kappa$, there is a unique polynomial $\phi_\kappa$ with integer coefficients such that for all polynomials $f(\underline{t})\in \cR=\Witt[[\underline{t}]]$, if we write $f(\underline{t}) = \sum_\alpha c_\alpha (1+\underline{t})^\alpha$ in the notation of Remark \ref{poly-rmk},
		then 
	$$(\theta^\kappa f)(\underline{t}) = \sum_\alpha \phi_\kappa(\alpha) c_\alpha (1+\underline{t})^\alpha.$$ 
	\end{defi}

From the description of the action of the differential operators described in Equation \eqref{eqn-determinant} together with Equation \eqref{eq-diffopdef}, we deduce the following corollary of the proof of Proposition \ref{cong1}.

	\begin{cor} \label{cordef}
	Let $\kappa$ be a sum-symmetric weight. Then 
	$$\phi_\kappa(\alpha)  
		 =\prod_{\tau\in\Sigma}\left(\left(\siga_\tau !\cdot m^\tau_{\siga_\tau}(\alpha)\right)^{\kappa_{\siga_\tau}^{\tau}}\prod_{i=1}^{\siga_\tau-1}\left(i!\cdot m^\tau_i(\alpha)\right)^{\kappa_{i}^{\tau}-\kappa_{i+1}^{\tau}}\right)$$  where $m^\tau_i(\alpha)$ is (a determinant of) an $i \times i$ minor of the matrix $\alpha$, for each $i$, $1\leq i\leq \siga_\tau$, and $\tau\in \Sigma$.
	 \end{cor}

	 \begin{rmk}\label{cordef2}
	 Let $\kappa$ and $\kappa'$ be two sum-symmetric weights satisfying the conditions of Proposition \ref{cong1}. Then $$\phi_\kappa(\alpha) \equiv \phi_{\kappa'}(\alpha) \mod p^{m+1}. $$
\end{rmk} 

We extend the definition of the polynomials $\phi_\kappa$ as follows. We write $\cO_{\bC_p}$ for the ring of integers of $\bC_p$,  the completion of an algebraic closure of $\bQ_p$.

\begin{defi}\label{cordef1} Let $\zeta:T(\zz_p)\ra \cO_{\bC_p}^*$ be any continuous group homomorphism. We write $\zeta=\prod_{\tau\in\Sigma}\left(\prod_{i=1}^n\zeta^\tau_i\cdot \varepsilon_i^\tau\right)$, where the $\zeta_{i}^{\tau}$ are continuous group homomorphisms $\bZ_p^*\ra \cO_{\bC_p}^*$ (possibly including finite order characters). We define
\begin{align}\label{phikap-rmk}
\phi_{\zeta}(\alpha) :=\prod_{\tau\in\Sigma}\left(\left(\zeta^\tau_{\siga_\tau}\left(\siga_\tau!\cdot m^\tau_{\siga_\tau}(\alpha)\right)\right)\prod_{i=1}^{\siga_\tau-1}\left(\zeta_i^\tau\cdot\left(\zeta_{i+1}^{\tau}\right)^{-1}\right)\left(i!\cdot m^\tau_i(\alpha)\right)\right),
\end{align}
where the $m^\tau_i(\alpha)$ are as in Corollary \ref{cordef}.
\end{defi}

 \begin{rmk}
	 It follows from the defintion that, if $\zeta,\zeta':T(\bZ_p)\ra  \cO_{\bC_p}^*$ are two continuous group homomorphisms satisfying $\zeta\equiv \zeta' \mod p^{m+1}$, then $\phi_\zeta(\alpha) \equiv \phi_{\zeta'}(\alpha) \mod p^{m+1}. $
\end{rmk}

\section{Pullbacks}\label{pullbacks-section}
In this section, we discuss the composition of the differential operators with pullbacks to a smaller group.  This construction is similar to the one in \cite[Section 3]{emeasurenondefinite}.  We further describe the action in terms of Serre--Tate coordinates (which are absent from \cite{emeasurenondefinite}), and then we obtain formulas (in terms of Serre--Tate coordinates) in the case of {\it all} signatures.  (Using $q$-expansions, \cite{emeasurenondefinite} had only obtained formulas when the signature at each archimedean place was one of just two possibilities.)  This section builds on \cite[Section 6]{CEFMV}, which provides details about pullbacks of automorphic forms in terms of Serre--Tate coordinates.

\subsection{Pullback and restriction of automorphic forms} \label{section-G-prime}

We start by introducing the required notation.  
Let $L=\oplus_{i=1}^s W_i$ be a self-dual $\cO_K$-linear decomposition of the free $\bZ$-module $L$. {For each $i$, $1\leq i\leq s$}, we denote by ${\langle, \rangle}_i$ the pairing on $W_i$ induced by $\langle,\rangle$ on $L$, and define
$GU_i=GU(W_i,\langle, \rangle_i)$, a unitary group of signature
$\left(\siga_{\tau}^{(i)},\sigb_{\tau}^{(i)}\right)_{\tau \in \Sigma_K}$.  
We write $\nu_i:GU_i\to\bG_m$ for the similitude factor.
 Note that the signatures $\left(\siga_{\tau}^{(i)},\sigb_{\tau}^{(i)}\right)_{i=1,\dots ,s}$ form a partition of the signature $\left(\siga_{\tau},\sigb_{\tau}\right)$. We define $G':=\nu_0^{-1}(\bG_m)\subset \prod_i GU_i$, where $\nu_0:=\prod_i\nu_i$, and $\bG_m\subset \bG_m^s$ is embedded diagonally.
Then, 
there is a canonical injective homomorphism $G'\hookrightarrow GU$, which induces 
a map $\phi$ between the 
associated moduli spaces, $\phi: \cM'\to\cM$. 
   Let $\cS'$ be a connected component of $\cM'{_\Witt}$.  We can identify $\cS'$ with the cartesian product of connected components of the smaller unitary Shimura varieties. 
We write $\cS$ for the unique connected component of $\cM_\Witt$ containing the image of $\cS'$, and we still denote by $\phi:\cS'\to\cS$ the restriction of $\phi.$  

We assume the prime $p$ splits completely over each of the reflex fields $E_i$, $i=1,\dots ,s$, associated with the smaller Shimura varieties, and let
{$\phi: \Sordprime\rightarrow\Sord$ also denote the restriction of $\phi$ to the ordinary loci.}
We denote respectively by $\Igusa '$, $\Ig$ the Igusa towers over $\Sordprime$, $\Sord$, and define 
$$H'=\prod\limits_{\tau \in \arch, 1 \leq i \leq s} \GL_{\siga_{\tilde\tau}^{(i)}} \times \GL_{\sigb_{\tilde\tau}^{(i)}}.$$
{We also write $H'=\prod_{1\leq i\leq s} H_i$, where $H_i= \prod_{\tau\in\Sigma} \GL_{\siga_{\tilde\tau}^{(i)}} \times \GL_{\sigb_{\tilde\tau}^{(i)}}$, for all $i=1, \dots ,s$.} The algebraic group $H'$  can be identified over $\bZ_p$ with a Levi subgroup of  $G' \cap U$.  Thus we have a closed immersion $\Levin' \ra \Levin$ {arising from} the inclusion of $G'$ into $GU$ and the identification over $\bZ_p$ of $H$ with a Levi subgroup of $U$.  This allows us (by choosing without loss of generality a suitable basis) to identify the maximal torus $T$ of $\Levin$ with a maximal torus $T'$ in $\Levin'$.  In the following, we denote by $X(T')_+$ the set of }the weights in $X(T') = X(T)$ that are dominant with respect to {the roots of $\Delta$ that belong to the root system of $H'$.} We also write ${V'}^{N'}$ for the space of $p$-adic automorphic forms on $H'$.

\subsubsection{Pullbacks}\label{epsilon}
In \cite[Proposition{~6.2}]{CEFMV} we observed that the map $\phi:\cS'{^{\rm{ord}}}\to\cS{^{\rm{ord}}}$ lifts canonically to a map between the Igusa covers, $\Phi:\Igusa '\to\Ig$.  As a consequence, we are able to explicitly {describe the pullback $\phi^*: V^N \ra {V'}^{N'}$} in the Serre--Tate coordinates associated with a point $x\in\Igusa '(\Witt)$ (and $\Phi(x)\in\Ig (\Witt)$).  To recall the result we first establish some notation.
For each $\tau\in \Sigma$, we denote by $\cL_\tau^\pm=\oplus_{i=1}^s\fL_{i,\tau}^\pm$ the associated $\zz_p$-linear decompositions of the modules $\cL_\tau^\pm$ ({arising from} the signature partition). 
We define $\fL^2:=\oplus_i\fL_i^2$, where for each $i=1,\dots ,s$
\[\fL^2_i:=\bigoplus_{\tau\in\Sigma}(\fL^+_{i,\tau}\otimes_{\zz_p} \fL^-_{i,\tau}).\]
In the following we denote by $\epsilon: \fL^2\hookrightarrow \cL^2$ the natural inclusion as a direct summand. 

We fix a point $x\in  \Igusa '(\Witt)$, and we write $\xzero\in\Sordprime(\Witt)$ for the point below $x$ (thus the point $\Phi(x){\in \Ig(\Witt)}$ is above $\phi(\xzero)\in\Sord(\Witt)$).  Following the notation of Section \ref{ST-coord}, we write 
\begin{eqnarray*}\beta_{\Phi(x)}^*:\Witt[[t]]\otimes (\cL^2)^\vee &\isomto &\cR:=\cR_{\Sord,\phi(\xzero)},\text{ and } \\
\beta_{x}^*:\Witt[[t]]\otimes (\fL^2)^\vee &\isomto& \cR'=\cR_{\Sordprime,\xzero},
\end{eqnarray*}
for the corresponding Serre--Tate isomorphisms of complete  local rings.

In \cite[Proposition{~6.8}]{CEFMV}, we prove that the ring homomorphism $\phi^*:\cR\to\cR'$ induced by the map $\phi:\Sordprime\to\Sord$ satisfies the equality
\begin{align}\label{phii-epsilon}
\phi^*\circ \beta_{\Phi(x)}^*=\beta_x^*\circ (\id\otimes \epsilon^\vee).
\end{align}

\begin{rmk} \label{proj}
For a choice of compatible bases $\fF\subset \fE$ of $\fL^2\subset \cL^2$, the pullback map on local rings $\cR \rightarrow \cR'$ described on coordinates as \[\id\otimes\epsilon^\vee: \Witt[[t]]\otimes (\cL^2)^\vee =\Witt[[t_l|l\in\fE^{\vee}]]\longrightarrow \Witt[[t]]\otimes (\fL^2)^\vee=\Witt[[t_l|l\in \fF^{\vee}]]\] satisfies the equalities 
\[ (\id\otimes\epsilon^\vee)(t_l)=\begin{cases} t_l\text{ if } l\in \fF^\vee\\ 0 \text{ otherwise}\end{cases}   \text{ for all } l\in \fE^\vee .\]
\end{rmk}

In the following,  with abuse of notations, we will identify $\cR'\simeq\Witt[[t_l|l\in \fF^{\vee}]]$ via $\beta^*_x$, and $\cR\simeq \Witt[[t_l|l\in\fE^{\vee}]]$ via $\beta^*_{\Phi(x)}$.

\subsubsection{Restrictions of $p$-adic automorphic forms}\label{restriction} 
Finally, we recall the definition of restriction on the space of $p$-adic automorphic forms. 

Let $\kappa,\kappa'$ be two characters of the torus $T'=T$. Assume $\kappa \in X(T)_+$, and $\kappa'\in X(T')_+$;  i.e., $\kappa$ is dominant for $H$, and $\kappa'$ is dominant for  $H'$. We say that $\kappa'$ {\em contributes to}
$\kappa$ if {$\rho_{\kappa'}$ is a} {quotient} of the restriction of $\rho_\kappa$ from $\Levi$ to $\Levi'$. 

 In the following, we denote by $\varpi_{\kappa,\kappa'}:\left.\rho_\kappa\right|_{H'}\to\rho_{\kappa'}$ a projection of $\Levi'_{\zz_p}$-representations.  If $\kappa' \in X(T')_+$ satisfies $(\kappa')^\sigma=\kappa$ for some $\sigma\in W_\Levi(T)$, then $\kappa'$ contributes to $\kappa$ and we choose $\varpi_{\kappa,\kappa'}: \left.\rho_\kappa\right|_{H'}\to\rho_{\kappa'}$ to be the projection of $\Levi'_{\zz_p}$-representations satisfying the equality $$\ell_{\rm can}^\kappa=\ell_{\rm can}^{\kappa'}\circ \varpi_{\kappa,\kappa'}\circ g_\sigma,$$ where $g_\sigma\in N_\Levin(T)(\zz_p)$ is the elementary matrix lifting  $\sigma$ (and $N_H(T)$ denotes the normalizer of $T$ in $H$).

\begin{rmk}\label{dominant}
If $\kappa$ is a dominant weight of $H$, then $\kappa$ is also dominant for  $H'$ and as a weight of $H'$ it contributes to $\kappa$. 
\end{rmk}

\begin{rmk}
{Assume $\kappa'$ is a weight of $H'$ contributing to a weight $\kappa$ of $H$. Then, for each weight $\lambda$ of $H$, the weight $\lambda\cdot \kappa'$ of $H'$ contributes to $\lambda\cdot\kappa$.}
\end{rmk}

For each $\kappa \in X(T)_+$,  let $\phi^*\cE_\kappa$ denote the pullback over $\cM'$   of the automorphic sheaf over $\cM$. To avoid confusion, we will denote by $\cE'_{\kappa'}$ the automorphic sheaf of weight $\kappa'$ over $\cM'$, for $\kappa' \in X(T')_+$.
For each weight $\kappa' \in X(T')_+$ contributing to $\kappa$, the morphism $\varpi_{\kappa,\kappa'}$ induces a morphism of sheaves over $\cM'$,
\[r_{\kappa,\kappa'}:\phi^*\cE_\kappa\to\cE'_{\kappa'}.\]

We define the {\em $(\kappa,\kappa')$-restriction} on automorphic forms to be the map of global sections
\[{\rm res}_{\kappa,\kappa'}:=r_{\kappa,\kappa'}\circ \phi^*:H^0(\cM,\cE_\kappa)\to H^0(\cM',\cE'_{\kappa'}).\]
By abuse of notation we will still denote by $\res_{\kappa,\kappa'}$ the map induced by $\res_{\kappa,\kappa'}$  between the spaces of sections of automorphic sheaves over the ordinary loci, i.e. 
\[{\rm res}_{\kappa,\kappa'}:H^0(\Sord,\cE_\kappa)\to H^0(\Sordprime,\cE'_{\kappa'}).\]
In the following, we write ${\rm res}_{\kappa}:={\rm res}_{\kappa,\kappa}$.

Finally, we define the {\em  restriction } on  $p$-adic automorphic forms as the pullback on global functions on the Igusa tower under $\Phi:\Igusa '\to\Igusa$,  i.e.
 \[\res:=\Phi^*:V^N\to {V'}^{N'}.\] 

In  \cite[Propositions{~6.5 and 6.6}]{CEFMV},  we compare the two notions of restriction.  Again, to avoid confusion, for all weights  $\kappa' \in X(T')_+$, we
denote by $\Psi'_{\kappa'}$ the inclusion  $H^0(\Sordprime,\cE'_{\kappa'})\to {V'}^{N'}$,  to distinguish it from the inclusion $\Psi_\kappa:  H^0(\Sord,\cE_{\kappa})\to V^N$ for $\kappa \in X(T)_+$.  

Then, for all dominant weights $\kappa \in X(T)_+$,
\begin{align}\label{res}
\res\circ \Psi_\kappa=\Psi'_\kappa\circ \res_\kappa.\end{align}
More generally,  if $\kappa' \in X(T')_+$ satisfies $(\kappa')^\sigma=\kappa$ for some $\sigma\in W_\Levi(T)$, then  
\begin{align}\label{10}
\Psi'_{\kappa'}\circ \res_{\kappa,\kappa'}=\res\circ (g_\sigma \cdot\Psi_\kappa).
\end{align}
for {$g_\sigma\in N_\Levi(T)(\zz_p)$} lifting $\sigma$ {as above}.

\subsection{Pullbacks and differential operators}
Let $\lambda \in X(T)= X(T')$. 
It follows from Definitions \ref{sum-sym-def} and \ref{sym-def} that if $\lambda$ is (sum-)symmetric for $H'$ and dominant for $H$, then it is also (sum-)symmetric for $H$, while the converse is false in general. In the following, we say that a weight of $H$ is {\em $H'$-(sum-)symmetric } if it is (sum-)symmetric for $H'$. 
Similarly, we say that a $p$-adic character of $H$ is $H'$-symmetric if it arises as the $p$-adic limit of $H'$-symmetric weights.

	For any weight (resp. $p$-adic character) $\lambda$ of $H'$, and for all $i=1,\dots ,s$, we write $\lambda_i:=\lambda_{| H_i}$ for the restriction of $\lambda$ to $H_i\subset H'$. Then, for each $i$, $\lambda_i$ is a weight (resp. $p$-adic character) of $H_i$. 
	We observe that a weight $\lambda$ is dominant (resp. sum-symmetric) for $H'$ if and only if, for each $i=1,\dots ,s$, the weight $\lambda_i$ is dominant (resp. sum-symmetric) for $H_i$.
	Furthemore, if, for each $i=1, \dots ,s$, $\lambda_i$ is sum-symmetric for $H_i$ of depth $e_i$, then $\lambda$ is sum-symmetric for $H$ of depth $e=\sum e_i$ if it is dominant. 

\begin{defi}
		A sum-symmetric weight $\lambda$ of $H'$ of depth $e$ is called {\em pure} if there exists $i\in\{1,\dots ,s\}$ such that $\lambda _i$ is sum-symmetric of depth $e$.
		Equivalently, a weight $\lambda$ is pure sum-symmetric if there exists 
		$i\in\{1,\dots ,s\}$ such that $\lambda _j$ is trivial for all $j\neq i$ and is sum-symmetric 
		for $j=i$. 
				
		Similarly,  a symmetric $p$-adic character $\chi$ of $H'$ is called {\em  pure } if  there exists $i\in\{1,\dots ,s\}$ such that $\chi_j:=\chi_{|H_j}$  is trivial for all $j\neq i$ and symmetric for $j=i$.
\end{defi}

\begin{rmk}
		If a sum-symmetric weight $\lambda$ of $H'$ is pure, of depth $e$,  then the associated irreducible representation $\rho_\lambda$ of $H'_{\bZ_p}$ is a quotient of $(\fL_i^2)^{\otimes e} $, for some $i\in\{1,\dots ,s\}$ (where $\oplus_i (\fL^2_i)^{\otimes e}$ is by definition a direct summand of $(\fL^2)^{\otimes e}$).
\end{rmk}

In the following, we say that a weight (resp. $p$-adic character) of $H$ is {\em pure $H'$-sum-symmetric} if it is a pure sum-symmetric for $H'$.

Note that if $\lambda$ is pure $H'$-sum-symmetric, with $\lambda_i$ non-trivial of depth $e$, then $\lambda$ is sum-symmetric of $H$ also of depth $e$, if it is dominant (see Remark \ref{counter}).

\begin{rmk}{
		A weight $\lambda$ of $H$ is both dominant for $H$ and pure $H'$-sum-symmetric if and only if  $\lambda_1$ is sum-symmetric for $H_1$ and $\lambda_j$ is trivial for all $j>1$.
		
		Similarly, a $p$-adic character $\chi$ of $H$ is both pure and $H'$-symmetric if and only if $\chi_1$ is a $p$-adic symmetric character of $H_1$ and $\chi_j$ is trivial for all $j>1$.}
\end{rmk}

For each $H'$-sum-symmetric weight $\lambda$ of $H$ {and positive dominant weight $\kappa'$ of $H'$}, we write
\[ \cD_{\kappa'}^{\lambda}:H^0(\Sordprime, \cE'_{\kappa'})\to H^0(\Sordprime ,\cE'_{\lambda\cdot \kappa'}),\]
for the associated differential operators on the automorphic forms of weight $\kappa'$. 

For each $H'$-symmetric $p$-adic {character} {$\chi$} of $H$, we write 
\[\Theta^{' \chi}: {V'}^{N'}\to {V'}^{N'}\]
for the corresponding operator on the space of $p$-adic automorphic forms (as in Theorem \ref{Theta}).

\begin{prop}\label{prop-pullback}
	For all pure $H'$-symmetric $p$-adic characters $\chi$ of $H$ that are $H$-symmetric 
	\[\res\circ\Theta^\chi =\Theta^{' \chi}\circ \res .\]
\end{prop}

\begin{proof}
	By the $t$-expansion principle, it suffices to verify the above equality after localizing at an ordinary point, i.e. as an equality of operators on $t$-expansions.
	Moreover, by Theorem \ref{congruence}, it is enough to consider the case {when $\chi$ has integral weight $\lambda$}. 
	
	In view of Remark \ref{proj} (together with the formulas in Lemmas \ref{formula-lemma} and \ref{theta}), the statement is a consequence of {the assumptions on $\lambda$ (which imply that the coefficients $a_{\lambda,\underline{l}}$ defined in Remark \ref{aipm} satisfy $a_{\lambda,\underline{l}}=0$ for all $\underline{l}\in \fE_e^{\vee}-\fF_e^\vee$)
		together with} the following general fact. For all positive integers $n, m\in \bN$, $n\geq m$,  the homomorphism of $\Witt$-algebras $\Witt[[t_1,\dots ,t_n]]\to\Witt[[t_1,\dots ,t_m]]$ defined as \[ f=f(t_1,\dots ,t_n)\mapsto f(t,0):=f(t_1,\dots, t_m,0,\dots 0)\]  satisfies the equalities  \[(1+t_i)\frac{\partial}{\partial t_i}(f(t,0))= ((1+t_i)\frac{\partial}{\partial t_i}f)(t,0),\]
	for all $i=1,\dots, m$.  
\end{proof}

\begin{rmk}\label{counter} 
		As an example, we consider the partition $(1,1), (1,1)$ of the signature $(2,2)$, and weights $\lambda=(2,0, 2,0)$ and $\lambda'=(1,1,1,1)$. Note that both $\lambda,\lambda'$ are dominant symmetric weights of both $H=GL(2)\times GL(2)$ and $H'=GL(1)\times GL(1)\times GL(1)\times GL(1)$,  but only $\lambda$ is $H'$-pure.
		Write $\Phi^*:\cR\simeq\Witt[[t_{1,3},t_{1,4},t_{2,3},t_{2,4}]]\ra \cR'\simeq\Witt[[T_{1,3},T_{2,4}]]$ for the map of complete local rings at an ordinary point corresponding to the inclusion of Igusa varieties. With our notations, $\Phi^*(t_{i,j})= T_{i,j}$ for $(i,j)=(1,3),(2,4)$  and $0$ otherwise.
		If we compute the associated differential operators $\theta^\lambda$ and $\theta^{\lambda'}$ of $\cR$, we obtain $\theta^\lambda=(\theta_{1,3})^2$ and $\theta^{\lambda'}=\theta_{1,3}\theta_{2,4}-\theta_{1,4}\theta_{2,3}$.
		If instead we compute the associated differential operators $\theta^{'\lambda}$ and $\theta^{'\lambda'}$on $\cR'$, we obtain  
		$\theta^{'\lambda}=(\theta'_{1,3})^2$ and $\theta^{'\lambda'}=\theta'_{1,3}\theta'_{2,4}$. It is easy to check that as maps on $\cR'$
		$\theta^{'\lambda}\circ \Phi^*=\Phi^*\circ \theta^\lambda$ but $\theta^{'\lambda'}\circ \Phi^*\neq \Phi^*\circ \theta^{\lambda'}$. 
		An explanation comes from the fact that the weight $\lambda$ has depth $2$ for both $H$ and $H'$, while $\lambda'$ should be regarded as of depth $2$ for $H$ and $1$ for $H'$.
	\end{rmk}

	We now consider the case of a pure symmetric $p$-adic character $\chi$ of $H'$, which is not a symmetric  $p$-adic character of $H$ (i.e., of $\chi$ arising as the $p$-adic limit of pure symmetric weights which are dominant for $H'$ but not for $H$).
	For such $p$-adic characters we have {already} defined a differential operator on automorphic forms on $H'$ but not on $H$, Proposition \ref{extend} explains how to extend it to $H$.

	Note that for any weight $\lambda$ of $H'$, there is a unique weight $\lambda_0$ which is dominant for $H$ and conjugate to $\lambda$ under the action of the Weyl group $W_H(T)$, i.e. $\lambda_0=\lambda^\sigma$ for some $\sigma \in W_H(T)$. 
	If $\lambda$ is a (sum-)symmetric weight of $H'$, then $\lambda_0$ is a (sum-)symmetric weight of $H$. Furthermore, if $\lambda$ is pure sum-symmetric for $H'$, then the permutation $\sigma\in W_H(T)$ can be chosen to arise from a permutation of 
	$\{1, \dots ,s\}$. More precisely, if $i\in\{1,\dots, s\}$ is such that $\lambda_j$ is trivial for all $j\neq i$, then we can choose $\sigma=\sigma_{(1i)}$ to correspond to the permutation $(1i)$. (Here, for each permutation $\gamma$ on $\{1,\dots ,s\}$, we define $\sigma_\gamma$ to be the element of $W_H(T)$ induced by the action of $\gamma$ on the partition $\{( \siga_\tau^{(j)},\sigb_\tau^{(j)})_{\tau\in\Sigma} |j=1, \dots ,s\}$ of the signature $(\siga_\tau,\sigb_\tau)_{\tau\in\Sigma}$.)
	In particular, we observe that $\lambda_0$ is pure sum-symmetric for the subgroup $H'_\sigma:=\prod_{1\leq j\leq s}H_{\gamma(j)}$ of $H$, for $\gamma=(1i)$.

		\begin{prop}\label{extend}
			
			For all pure symmetric $p$-adic characters $\chi$ of $H'$, let  $\sigma=\sigma_{(1i)}$, for $i\in\{1,\dots ,s\}$ such that $\chi_j$ is trivial for all $j\neq i$. Then $\chi^\sigma$ is a symmetric $p$-adic character of $H$ and 
			\[ \Theta^{' \chi}\circ \res =\res \circ (g_\sigma \circ \Theta^{\chi^\sigma}\circ  g_\sigma^{-1}).\]
		\end{prop}
		\begin{proof}
			As  in the proof of Proposition \ref{prop-pullback}, we use Theorem  \ref{congruence} to reduce to the case when $\chi$ has integral weight $\lambda$, 
			and by the $t$-expansion principle it suffices to establish the equality after localization at a point $x$ of the smaller Igusa tower $\Igusa '$. 
			For all $f\in V^N$, on one side we have
			\[ (\Theta^{' \lambda}\circ \res)(f)_x(t) =\theta^{'\lambda}(\res(f)_x(t))=\theta^{' \lambda}\circ \Phi^* (f_{\Phi(x)}(t)),\]
			and on the other  
			\begin{eqnarray*}(\res \circ (g_\sigma \circ \Theta^{\lambda^\sigma}\circ  g_\sigma^{-1}))(f)_x(t) &=&
			\Phi^*((g_\sigma \circ \Theta^{\lambda^\sigma}\circ  g_\sigma^{-1})(f)_{\Phi(x)}(t))\\ &=&
			(\Phi^*\circ g_\sigma)((\Theta^{\lambda^\sigma}\circ  g_\sigma^{-1})(f)_{\Phi(x)^{g_\sigma}}(t))\\ 
			&=& (\Phi^*\circ g_\sigma\circ \theta^{\lambda^\sigma})((g_\sigma^{-1})(f)_{\Phi(x)^{g_\sigma}}(t))\\ &=& (\Phi^*\circ g_\sigma\circ \theta^{\lambda^\sigma}\circ g_\sigma^{-1})((f)_{\Phi(x)}(t)).
			\end{eqnarray*}
			Thus, we have reduced the statement to an equality of two maps $\cR\simeq \Witt[[t_l|l\in \fE^\vee]]\ra \cR'\simeq \Witt[[t_l|l\in \fF^\vee]]$, i.e. \[\theta^{' \lambda}\circ \Phi^*=\Phi^*\circ g_\sigma\circ \theta^{\lambda^\sigma}\circ g_\sigma^{-1},\]
			where $\Phi^*={\rm id}\otimes\epsilon^\vee$  is the map on complete local rings corresponding to the map $\Phi$ between Igusa varieties described in Remark \ref{proj}.
			Recall that the action of $g_\sigma$ on $\Witt[[t_l|l\in \fE^\vee]] $ is given by the formula $g_\sigma(t_l)=t_{\sigma(l)}$ for all $l\in\fE^\vee$.
			Write $\theta^{\lambda^\sigma}= \sum_{\underline{l}\in\fE^\vee_e} a_{\kappa,\underline{l}}\cdot \theta^{\ul d(\underline{l})}$ and define $\theta^\lambda:=\sum_{\underline{l}\in\fE^\vee_e} a_{\kappa,\underline{l}}\cdot \theta^{\ul d(\sigma(\underline{l}))}$ (note that here $\lambda$ is possibly not dominant for $H$).
			Then, on the right hand side, we have
			\[\Phi^*\circ g_\sigma\circ \theta^{\lambda^\sigma}\circ g_\sigma^{-1} = \Phi^*\circ  \theta^{\lambda}\circ g_\sigma\circ g_\sigma^{-1}= \Phi^*\circ  \theta^{\lambda}. \]
	Finally, the same computation as in the proof of Proposition \ref{prop-pullback} implies
			$ \theta^{' \lambda}\circ \Phi^*=\Phi^*\circ  \theta^{\lambda}$.				
			\end{proof}

	\begin{rmk}
		Given a pure $H'$-(sum)-symmetric dominant weight $\lambda $ of $H$, Proposition \ref{prop-pullback} together with Equation \eqref{res} imply the equality
		\[\res_{\lambda\cdot \kappa}\circ D_\kappa^\lambda=\cD_\kappa^\lambda\circ \res_\kappa,\] 
		for all positive dominant weights $\kappa$ of $H$.  In fact, the same argument, combined with Equation (\ref{10}), also proves
		\[{\rm res}_{\lambda\cdot \kappa, \lambda\cdot \kappa'}\circ D^\lambda_\kappa= {\cD}_{\kappa'}^{\lambda}\circ {\rm res}_{\kappa,\kappa'},\] 
		for all dominat weights $\kappa$ of $H$, and $\kappa'$ of  $H'$,  such that $(\kappa')^\sigma=\kappa$ for some  $\sigma\in W_H(T)$,  assuming $\lambda^\sigma=\lambda$.
		Similarly, given a pure sum-symmetric weight $\lambda$ of $H'$, if $i\in\{1,\dots ,s\}$ is such that $\lambda_0=\lambda^{\sigma_{(1i)}}$ is a (sum-symmetric) dominant weight of $H$, then Proposition \ref{extend} and Equation (\ref{10}) together imply that
		\[\res_{\lambda_0\cdot\kappa,\lambda\cdot\kappa}\circ D^{\lambda_0}_{\kappa}\circ g_\sigma^{-1}= \cD^{\lambda}_{\kappa} \circ \res_{\kappa},\]
		for all positive dominant weights $\kappa$ of $H$, satisfying $\kappa^{\sigma_{(1i)}}=\kappa$.
	\end{rmk}

\begin{rmk} Let $\kappa,\kappa',\lambda,\lambda_0$ and $\sigma,\sigma_{(1i)}$ be as above.
The choice of projection $\varpi_{\kappa,\kappa'}: \left.\rho_\kappa\right|_{H'}\to\rho_{\kappa'}$, satisfying 
$\ell_{\rm can}^\kappa=\ell_{\rm can}^{\kappa'}\circ \varpi_{\kappa,\kappa'}\circ g_\sigma$, uniquely determines one quotient of 
$\rho_\kappa|_{H'}$ of weight $\kappa'$, even when the multiplicity of the $\Levi'_{\zz_p}$-representation $\rho_{\kappa'}$ in $\rho_\kappa|_{H'}$ (by which we mean the rank of $\Hom_{\Levi'_{\zz_p}}(\rho_\kappa|_{H'},\rho_{\kappa'})$) is larger than 1. 
For a general $\lambda$,  the multiplicities of $\rho_{\lambda\kappa'}$ in $\rho_{\lambda\kappa}|_{H'}$ and of  $\rho_{\kappa'}$ in $\rho_{\kappa}|_{H'}$ 
(resp. of $\rho_{\lambda_0\kappa}$ in $\rho_{\lambda\kappa}|_{H'}$ and of  $\rho_{\kappa}$ in $\rho_{\kappa}|_{H'}$) 
 might not agree. However, our (uniform) choice of projections ensures the
 compatibility of the resulting restrictions of automorphic forms with
 the differential operators. 
\end{rmk}

\section{$p$-adic families of automorphic forms}\label{families-section}\label{appli}

In this section, we construct $p$-adic families of automorphic forms on unitary groups.  To construct the families, we apply the differential operators introduced above to the Eisenstein series constructed in \cite{apptoSHLvv, apptoSHL} and then apply Theorem \ref{congruence}.  We also construct a $p$-adic measure by applying the description of the differential operators in Section \ref{congruences-actions-section} (especially Equation \eqref{phikap-rmk}).

\subsection{Prior results on families for signature $(n,n)$}
We begin by recalling the Eisenstein series in \cite{apptoSHLvv}, which include the Eisenstein series in \cite{kaCM, apptoSHL} as special cases.  Similarly to the notation in \cite{apptoSHLvv}, for $k\in \ZZ$ and $\nu = \left(\nu(\sigma)\right)_{\sigma\in\Sigma}\in\ZZ^\Sigma$, we denote by $\mathbf{N}_{k, \nu}$ the function
\begin{equation*}
\mathbf{N}_{k, \nu}: \cmfield^\times \rightarrow \cmfield^\times \qquad b \mapsto \prod_{\sigma\in\Sigma}\sigma(b)^{k}\left(\frac{\sigma(b)}{\overline{\sigma}(b)}\right)^{\nu(\sigma)}.
\end{equation*}
Note that for all $b\in\mathcal{O}_{\realfield}^\times$, $\mathbf{N}_{k, \nu}(b) = \mathbf{N}_{\realfield/\IQ}(b)^k$.

The theorem below gives explicit $q$-expansions of automorphic forms.  Note that as explained in \cite[Section 8.4]{hida}, to apply the $p$-adic $q$-expansion principle in the case of a unitary group $U(n,n)$ of signature $(n,n)$ at each place (for some integer $n$), it is enough to check the cusps parametrized by points of $\GMplus\left(\adeles_{\realfield}\right)$, where $GM_+$ denotes a certain Levi subgroup of $U(n,n)$.  (More details about cusps appear in \cite[Chapter 8]{hida} and, as a summary, in \cite{apptoSHLvv}; we will not need the details here.)

\begin{thm}[Theorem 2 in \cite{apptoSHLvv}]\label{thm2}
Let $R$ be an $\OK$-algebra, let $\nu = \left(\nu(\sigma)\right)\in\ZZ^\Sigma$, and let $k\geq n$ be an integer.    Let
\begin{align*}
F: \left(\OK\otimes\ZZ_p\right)\times M_{n\times n}\left(\Oreal\otimes\ZZ_p\right)\rightarrow R
\end{align*}
be a locally constant function supported on $\left(\OK\otimes\ZZ_p\right)^\times\times \GL_n\left(\Oreal\otimes\ZZ_p\right)$ that satisfies
\begin{align}\label{equnknualakacm}
F\left(ex, \mathbf{N}_{K/E}(e^{-1})y\right) = \mathbf{N}_{k, \nu}(e)F\left(x, y\right)
\end{align}
for all $e\in \OK^\times$, $x\in \OK\otimes\ZZ_p$, and $y\in M_{n\times n}\left(\Oreal \otimes\ZZ_p\right)$.  There is an automorphic form $G_{k, \nu, F}$ of parallel weight $k$ on $GU(n,n)$
defined over $R$ whose $q$-expansion at a cusp $m\in GM_+\left(\adeles_{\realfield}\right)$ is of the form $\sum_{0<\alpha\in L_m}c(\alpha)q^\alpha$ (where $L_{m}$ is a lattice in $\hern(\cmfield)$ determined by $m$), with $c(\alpha)$ a finite $\ZZ$-linear combination of terms of the form
\begin{align*}
F\left(a, \mathbf{N}_{K/E}(a)^{-1}\alpha\right)\mathbf{N}_{k, \nu}\left(a^{-1}\det\alpha\right)\mathbf{N}_{E/\IQ}\left(\det\alpha\right)^{-n}
\end{align*}
(where the linear combination is a sum over a finite set of $p$-adic units $a\in \cmfield$ dependent upon $\alpha$ and the choice of cusp $m$).\end{thm}

Let $\left(R, \iota_\infty\right)$ consist of an $\cO_{E',(p)}$-algebra $R$ together with a ring inclusion $\iota_\infty: R\hookrightarrow \IC$.  Given an automorphic form $f$ defined over $R$, we view $f$ as a $p$-adic automorphic form via $\Psi$ or as a $\ci$-automorphic form after extending scalars via $\iota_\infty: R\hookrightarrow \IC$. 

\begin{rmk}
The $\ci$-automorphic forms $G_{k, \nu, F}$ are closely related to the $\ci$-Eisenstein series studied by Shimura in \cite{sh}; the difference between Shimura's Eisenstein series and these ones is the choice of certain data at $p$, which allows one to put $G_{k, \nu, F}$ into a $p$-adic family.  For $R = \IC$, these are the Fourier coefficients at $s=\frac{k}{2}$ of certain $\ci$-automorphic forms $G_{k, \nu, F}\left(z, s\right)$ (holomorphic in $z$ at $s=\frac{k}{2}$) of parallel weight $k$ defined in \cite[Lemma 9]{apptoSHLvv}.  We do not need further details about those $\ci$-Eisenstein series for the present paper.
  \end{rmk}

\subsection{Families for arbitrary signature} We use the notations introduced in Section  \ref{section-G-prime}, with $GU=GU(n,n)$ and $G'$ of arbitrary signature. In particular, we still denote by $H$ and $H'$ the associated Levi subgroups.

For each symmetric weight $\kappa$ of $H'$, we define an automorphic form of weight ${\underline{k}}+\kappa$ for $G'$
\begin{align*}
G_{k, \nu, F, \kappa}:=  \Theta^\kappa \res_{\underline{k}} G_{k, \nu, F},
\end{align*}
where $\res_{\underline{k}}$ is the restriction on automorphic forms from $GU(n,n)$ to $G'$,
(i.e. the pullback followed by projection onto an irreducible quotient) 
as introduced in Section \ref{section-G-prime}.

\begin{rmk}\label{pure-res}
Proposition \ref{prop-pullback}  implies that if the weight $\kappa$ is pure $H'$-symmetric and $H$-symmmetric, then the automorphic form $G_{k, \nu, F, \kappa}$ agrees with $ \res \Theta^\kappa G_{k, \nu, F}$, the restriction to $G'$ of the form $\Theta^\kappa G_{k, \nu, F}$ for $U(n,n)$.
\end{rmk}

As an immediate consequence of Theorem \ref{congruence} 
and Remark \ref{proj} 
applied to $\res_k G_{k, \nu, F}$, we obtain the following result.
\begin{thm}\label{fam-thm}
Let $\kappa$ and $\kappa'$ be two symmetric weights satisfying the conditions of Proposition \ref{congruence}.  Then $G_{k, \nu, F, \kappa}\equiv G_{k, \nu, F, \kappa'}$ inside $V'^{N'}/p^{m+1}V'^{N'}$.
\end{thm}

Proposition \ref{ordCM-prop} below summarizes the relationship between the values at CM points (together with a choice of trivialization) of the $p$-adic automorphic forms obtained by applying $p$-adic differential operators to $\res_{\underline{k}} G_{k, \nu, F}$ and the values of $\ci$ automorphic forms obtained by applying $\ci$ differential operators to $\res_{\underline{k}} G_{k, \nu, F}$.

\begin{prop}\label{ordCM-prop} 
For each locally constant function $F$ as in Theorem \ref{thm2}, the values of $G_{k, \nu, F, \kappa}$ and $\ellcan^{\underline{k}\cdot\kappa}\circ \pi^\infty_{\kappa\cdot\underline{k}}\circ D_{\underline{k}}^{\kappa}\left(\ci\right) \res_{\underline{k}} G_{k, \nu, F}$, where $\pi^\infty_{\kappa\cdot\underline{k}}$ denotes the projection onto an irreducible subspace of highest weight $\kappa\cdot\underline{k}$,
agree at each ordinary CM point $A$ over $R$ (together with a choice of trivialization of $\uo_A/R$) up to a period.  
\end{prop}
Thus, as a consequence of Theorem \ref{fam-thm}, we can $p$-adically interpolate the values of $\ellcan^{\underline{k}\cdot\kappa}\circ \pi^\infty_{\kappa\cdot\underline{k}}\circ D_{\underline{k}}^{\kappa}\left(\ci\right) \res_{\underline{k}} G_{k, \nu, F}$ (modulo periods) at ordinary CM points as $\kappa$ varies $p$-adically.
\begin{proof}
The proof is similar to \cite[Section 5]{kaCM}, \cite[Section 3.0.1]{apptoSHL}, and \cite[Section 5.1.1]{apptoSHLvv}.
\end{proof}

Let $\chi=\prod_w\chi_w$ be a Hecke character of type $A_0$.  We obtain a $p$-adically continuous character $\tilde{\chi}$ on $X_p$, where $X_p$ denotes the projective limit of the ray class groups of $\cmfield$ of conductor $p^r$, as follows.    Let $\tilde\chi_\infty: \left(\cmfield\otimes\Z_p\right)^\times\rightarrow \overline\IQ_p^\times$ be the $p$-adically continuous character such that
\begin{align*}
\tilde\chi_\infty (a) = \imath_p\circ \chi_\infty(a)
\end{align*}
for all $a\in \cmfield$.  So the restriction of $\tilde\chi_\infty$ to $\left(\cO_\cmfield\otimes\Z_p\right)^\times$ is a $p$-adic character.  We define a $p$-adic character $\tilde{\chi}$ on $X_p$
by $\tilde\chi\left(\left(a_w\right)\right) = \tilde\chi_\infty\left(\left(a_w\right)_{w\divides p}\right)\prod_{w\ndivides \infty}\chi_w(a_w)$.

For each character $\zeta$ on the torus $T$ and for each type $A_0$ Hecke character $\chi = \chi_u|\cdot|^{-k/2}$ with $\chi_u$ unitary, we define $F_{\chi_u, \zeta}(x, y)$ on $\left(\OK\otimes\ZZ_p\right)^\times\times \GL_n\left(\Oreal\otimes\ZZ_p\right)$ by $F_{\chi_u, \zeta}(x, y):=\chi_u(x)\phi_{\zeta}(\mathbf{N}_{K/E}(x){ }^ty)$ and extend by $0$ to a function on $\left(\OK\otimes\ZZ_p\right)\times M_{n\times n}\left(\Oreal\otimes\ZZ_p\right)$, with $\phi_\zeta$ defined as in Equation \eqref{phikap-rmk}.

We now construct a certain $p$-adic measure.

\begin{thm}\label{measurethm}
There is a measure $\Eisab$ (dependent on the signature of the group $G'$) on $X_p\times T\left(\ZZ_p\right)$ that takes values in the space of $p$-adic modular forms on $G'$ and that satisfies 
\begin{align}
\int_{X_p\times T\left(\ZZ_p\right)}\tilde{\chi}\psi \kappa\Eisab = \res \Theta^\kappa G_{k, \nu, F_{\chi_u, \psi}} \label{equ-diffopaction}
\end{align}
for each finite order character $\psi$ and $H$-symmetric weight $\kappa$ on the torus $T$ and for each type $A_0$ Hecke character $\chi = \chi_u|\cdot|^{-k/2}$ of infinity type $\prod_{\sigma\in\Sigma}\sigma^{-k}\left(\frac{\bar{\sigma}}{\sigma}\right)^{\nu(\sigma)}$ with $\chi_u$ unitary.

In particular, for $\kappa$ pure $H'$-symmetric, we have \[
\int_{X_p\times T\left(\ZZ_p\right)}\tilde{\chi}\psi \kappa\Eisab =G_{k, \nu, F_{\chi_u, \psi},\kappa}.
\]

 \end{thm}
 
 Equation \eqref{equ-diffopaction} is analogous to \cite[Equations (5.5.7)]{kaCM}, which concerns the case of an Eisenstein measure for Hilbert modular forms.  Theorem \ref{measurethm} also extends the main results of \cite{apptoSHL, apptoSHLvv} to arbitrary signature.

The idea of the proof is similar to the idea of the construction of analogous Eisenstein measures in \cite{kaCM, apptoSHL}, i.e. it relies on the $p$-adic $q$-expansion principle.
 \begin{proof} 
First, note that the measure $\mu_{G'}$ is uniquely determined by restricting to finite order characters on $X_p\times T\left(\ZZ_p\right)$ (by, for example, \cite[Proposition (4.1.2)]{kaCM}).  Now, Equation \eqref{equ-diffopaction} follows from the $p$-adic $q$-expansion principle (\cite[Corollary 10.4]{hi05}), as follows:  First, note that the $p$-adic $q$-expansion principle holds for all elements in $V_{\infty,\infty}$, not just those in $V^N$.  Now we apply the differential operators from Section \ref{Dwork-sec} to the automorphic form $G_{k, \nu, F_{\chi_u, \psi}}$ on the general unitary group $G$ of signature $(n,n)$.  So the resulting automorphic form takes values in a vector space that is a representation of $H$.  We project the image onto an irreducible representation for $H'$.  So the pullback of this automorphic form to $G'$ is $\res \Theta^\kappa G_{k, \nu, F_{\chi_u, \psi}}$.
Note that the action of differential operators on $q$-expansions is similar to the action on Serre--Tate expansions, with the parameter $\left(1+\underline{t}\right)$ replaced by $\underline{q}$.  (In each case, it depends on the existence of a horizontal basis.  See \cite[Corollary (2.6.25)]{kaCM} for the case of Hilbert modular forms, which is extended to unitary groups of signature $(n,n)$ in \cite{EDiffOps}.)  In particular, we see from Definition \ref{phikap-action} and Corollary \ref{cordef} that applying $D_{\underline{k}}^{\kappa}$ and then projecting the image onto a highest weight vector results in multiplying each $q$-expansion coefficient by the polynomial from Corollary \ref{cordef}.  Equation \eqref{equ-diffopaction} then follows immediately from Remark \ref{cordef2} together with the abstract Kummer congruences (see \cite[Proposition (4.0.6)]{kaCM}), i.e. the observation that for each integer $m$, whenever a linear combination of the values of the product of characters on the left hand side is $0\mod p^m$, then the corresponding linear combination of $q$-expansions of the automorphic forms on the right hand side is also $0\mod p^m$.

The statement for pure  weights $\kappa$ follows from Remark \ref{pure-res}
\end{proof}

\section{Acknowledgements}
We would like to thank Ana Caraiani very much for contributing to initial conversations about topics in this paper.  We would also like to thank the referee for a careful reading and helpful suggestions.   We are grateful to have had the opportunity to meet at Caltech in June 2014 and February 2015.  Most of the discussions that led to this paper took place there.  We are also grateful to have had the opportunity to make progress on this work while at MSRI in August 2014.

\bibliography{WIN3bib}

\end{document}